\documentclass[12pt,a4paper,reqno]{amsart}

\usepackage{amsmath}
\usepackage{amsfonts}
\usepackage{palatino}
\usepackage{a4wide}
\usepackage{amssymb}
\usepackage{amsthm}
\usepackage{stmaryrd}
\usepackage{mathrsfs}

\usepackage[centertags]{amsmath}
\usepackage{eucal,dsfont,accents,bbm,amsfonts,amssymb,amsthm,amsopn,verbatim}

\usepackage[usenames]{color}
\definecolor{citegreen}{rgb}{0,0.6,0}
\definecolor{refred}{rgb}{0.8,0,0}
\usepackage[colorlinks, citecolor=blue, linkcolor=refred]{hyperref}

\usepackage[]{draftcopy}
\draftcopyName{ Draft}{240}
\draftcopySetScale{270}

\theoremstyle{plain}

\newtheorem{teo}{Theorem}[section]
\newtheorem{lemma}[teo]{Lemma}
\newtheorem{prop}[teo]{Proposition}

\newtheorem{ackn}{Acknowledgments\!}

\theoremstyle{definition}
\newtheorem{dfnz}[teo]{Definition}

\theoremstyle{remark}

\newtheorem{rem}[teo]{Remark}

\numberwithin{equation}{section}

\def\SS{{{\mathbb S}}}
\def\NN{{{\mathbb N}}}

\def\RR{{\mathbb R}}

\def\R{{\mathbb R}}

\def\BBB{{\mathrm B}}

\def\Ric{{\mathrm{Ric}}}

\def\CCC{{\mathrm C}}
\def\Rm{{\mathrm{Riem}}}
\def\Pic{{\mathrm{Pic}}}
\def\PPP{{\mathrm{P}}}
\def\Pm{{\mathrm{Piem}}}
\def\WWW{{\mathrm W}}

\def\Sec{\mathrm {Sec}}

\def\dt{\frac{\partial\,}{\partial t}}

\def\tr{\operatornamewithlimits{tr}\nolimits}

\def\RRR{\mathrm{R}}
\def\Id{\mathrm {Id}}

\newcommand{\Derpar} [1] { \frac{\partial~ } {\partial {#1} }}
\newcommand{\derpar}[2] { \frac{\partial{#1}}{\partial {#2} }} 

\newcommand{\M}[1]{\mathcal{#1}}

\title[The Ricci--Bourguignon Flow]{The Ricci--Bourguignon Flow}

\author[G. Catino]{Giovanni Catino}
\address[Giovanni Catino]{Dipartimento di Matematica, Politecnico di Milano, Piazza Leonardo da Vinci 32, Milano, Italy, 20133}
\email[G. Catino]{giovanni.catino@polimi.it}

\author[L. Cremaschi]{Laura Cremaschi}
\address[Laura Cremaschi]{Scuola Normale Superiore di Pisa, Piazza dei Cavalieri 7, Pisa, Italy, 56126}
\email[Laura Cremaschi]{laura.cremaschi@sns.it}

\author[Z. Djadli]{Zindine Djadli}
\address[Zindine Djadli]{Institut Fourier, 100 Rue des Maths, BP 74, St.~Martin~d'Heres, France, 38402 and Laboratorio Fibonacci, Piazza dei Cavalieri 7, Pisa, Italy, 56126}
\email[Z. Djadli]{Zindine.Djadli@ujf-grenoble.fr}

\author[C. Mantegazza]{Carlo Mantegazza}
\address[Carlo Mantegazza]{Dipartimento di Matematica, Universit\`a di
  Napoli ``Federico II'', Via Cintia, Monte S.~Angelo, Napoli, Italy,
  I--80126}
\email[C. Mantegazza]{c.mantegazza@sns.it}

\author[L. Mazzieri]{Lorenzo Mazzieri}
\address[Lorenzo Mazzieri]{Universit\`a degli Studi di Trento, Via Sommarive 15, Povo, Trento, Italy, I--38123}
\email[L. Mazzieri]{lorenzo.mazzieri@unitn.it}

\date{\today}

\begin{document}

\begin{abstract}
In this paper we present some results on a family of geometric flows introduced by 
J.~P.~Bourguignon in~\cite{jpb1} that generalize the Ricci flow. For suitable values of the scalar parameter involved in these flows, we prove short time existence and provide curvature estimates. We also state some results on the associated solitons. 
\end{abstract}

\maketitle
\setcounter{tocdepth}{1}
\tableofcontents

\section{Introduction}

In this paper we consider an $n$--dimensional, compact, smooth,
Riemannian manifold $M$ (without boundary) whose metric $g=g(t)$ is evolving according
to the flow equation
\begin{equation}\label{eqflow}
\dt g \,=\, -2\,\Ric +2\rho \RRR g =  -2(\Ric -\rho \RRR g) \,
\end{equation}
where $\Ric$ is the Ricci tensor of the manifold, $\RRR$ its scalar
curvature and $\rho$ is a real constant. This family of geometric 
flows contains, as a
special case, the Ricci flow, setting $\rho=0$. Moreover, by a suitable
rescaling in time, when $\rho$ is nonpositive, 
they can be seen as an interpolation between the Ricci
flow and the Yamabe flow (see~\cite{brendle1,struwschwe,ye1}, for instance), obtained as a limit when
$\rho\to-\infty$.

It should be noticed that for special values of the constant $\rho$ the
tensor $\Ric -\rho\RRR g$ appearing on the right hand side of the evolution 
equation is of special interest, in particular,
\begin{itemize}

\item $\rho=1/2$, the Einstein tensor $\Ric -\frac{\RRR }{2}g$,

\item $\rho=1/n$, the traceless Ricci tensor $\Ric -\frac{\RRR }{n}g$,

\item $\rho=1/2(n-1)$, the Schouten tensor $\Ric -\frac{\RRR }{2(n-1)}g$,

\item $\rho=0$, the Ricci tensor $\Ric$.

\end{itemize}  
In dimension two, the first three tensors are zero hence the flow is
static, and in higher dimension the values of $\rho$
are strictly ordered as above, in descending order.

Apart these special values of $\rho$, for which we will call the
associated flows as the name of the corresponding tensor, 
in general we will refer to the evolution equation defined by the PDE
system~\eqref{eqflow} as the  Ricci--Bourguignon flow (or shortly RB flow).

The study of these flows was proposed by
Jean--Pierre Bourguignon in~\cite[Question~3.24]{jpb1}, building on some unpublished work of Lichnerowicz in the sixties and a paper of Aubin~\cite{aubin}. 
In 2003, Fischer~\cite{fischer} studied a conformal version of this problem where the scalar curvature is constrained along the flow. In 2011, 
Lu, Qing and Zheng~\cite{LQZ} also proved some results on the conformal Ricci--Bourguignon flow. Some results concerning solitons of the Ricci--Bourguignon flow (called {\em gradient $\rho$--Einstein solitons}) can be found in~\cite{CaMa, CaMaMo}.

We will see in the next section that when $\rho$ is larger than
$1/2(n-1)$ the principal symbol of the operator in the right
hand side of the second order quasilinear parabolic PDE~\eqref{eqflow}
has negative eigenvalues, not allowing even a short time existence
result for the flow for general initial data (manifold $M$ and
initial metric $g_0$). On the contrary, the main task of
Section~\ref{short} will be to show that for any $\rho<1/2(n-1)$,
every initial compact Riemannian manifold $(M,g_0)$ has a unique smooth
solution $g(t)$ solving the flow equation~\eqref{eqflow}, with
$g(0)=g_0$, at least in a positive time interval.

However, the problem of knowing whether the ``critical''
{\em Schouten flow}
\begin{equation}\label{schflow}
\begin{cases}
\displaystyle{\,\,\dt g \,=\, -2\,\Ric + \frac{\RRR }{n-1}}g\\
\displaystyle{\,\,g(0)=g_0}
\end{cases}
\end{equation}
when $\rho=1/2(n-1)$, admits or not a short time solution
 for general initial manifolds and metrics remains open, when $n\geq 3$. 

We will see that if $\rho\leq 1/2(n-1)$, the principal symbol of the elliptic operator is nonnegative
definite and it actually contains some zero eigenvalues due 
to the diffeomorphism invariance of the geometric flow. When 
$\rho<1/2(n-1)$, these zero eigenvalues are the only ones, all the
others are actually positive, hence, they can be dealt (as it is customary by
now) by means of the so--called  DeTurck's trick~\cite{deturck,deturck2}. In the case of
the Schouten flow $\rho=1/2(n-1)$ instead, the principal symbol 
contains an extra zero eigenvalue besides the ones due
to the diffeomorphism invariance, preventing this argument to be
sufficient to conclude and to give a general short time existence result.\\
We mention that the presence of this extra zero eigenvalue 
should be expected, as the Cotton tensor, which is obtained from the Schouten tensor
as follows
$$
\CCC_{ijk} = \nabla_{k} \RRR_{ij} - \nabla_{j} \RRR_{ik}=\nabla_k\RRR_{ij} - \nabla_j\RRR_{ik} -
\frac{1}{2(n-1)} \big(\nabla_k\RRR g_{ij} - \nabla_j\RRR g_{ik} \big)\,,
$$
satisfies the following invariance under the conformal change of
metric $\widetilde{g} = e^{2u} g$,
$$
e^{3u} \widetilde{\CCC}_{ijk} = \CCC_{ijk} + (n-2) \WWW_{ijkl}  \nabla^l  u\,
$$
see~\cite[Equation~3.35]{CatMasMonRig}. Recently, Delay~\cite{delay}, following the work of Fischer and Marsden, gave some evidence on the fact that the DeTurck's trick should fail for the Schouten tensor.

In Section~3, we will compute the evolution equations for the
curvature. 

In Section~4, by means of the maximum principle, we derive, from the evolution of the curvature, some 
conditions on the curvature which are preserved by the RB flow. In particular, we show that the Hamilton--Ivey estimate in dimension three holds true.

In Section~5, we establish some \emph{a priori} estimates on the Riemann tensor and prove that, if a compact solution of the flow exists up to a finite maximal time $T$, then the Riemann tensor is unbounded when approaching to $T$.

Finally, in the last section we discuss the structure and the classification of the solitons of the RB flow.

\medskip

\begin{ackn} We want to thank Mauro Carfora for several suggestions
  and comments. The authors are partially supported by the GNAMPA
  project ``Equazioni di evoluzione geometriche e strutture di tipo Einstein''. Laura Cremaschi is partially supported by the PRIN 2010--2011 project ``Calcolo delle Variazioni''. Zindine Djadli is partially supported by the grant GTO: ANR-12-BS01-0004.
\end{ackn}

\subsection{Notation and preliminaries}\ \\
The Riemann curvature operator of an $n$--dimensional Riemannian manifold $(M,g)$ is defined 
as in~\cite{gahula} by
$$
\mathrm{Riem}(X,Y)Z=\nabla_{Y}\nabla_{X}Z-\nabla_{X}\nabla_{Y}Z+\nabla_{[X,Y]}Z\,,
$$
and we will denote with $d\mu_g$ the canonical volume measure
associated to the metric $g$.\\
In a local coordinate system the components of the $(3,1)$--Riemann 
curvature tensor are given by
$\RRR^{l}_{ijk}\tfrac{\partial}{\partial
  x^{l}}=\mathrm{Riem}\big(\tfrac{\partial}{\partial
  x^{i}},\tfrac{\partial}{\partial
  x^{j}}\big)\tfrac{\partial}{\partial x^{k}}$ and we denote by
$\RRR_{ijkl}=g_{lm}\RRR^{m}_{ijk}$ its $(4,0)$--version.\\
With this choice, for the sphere $\SS^n$ we have
${\mathrm{Riem}}(v,w,v,w)=\RRR_{ijkl}v^iw^jv^kw^l>0$.

The Ricci tensor is obtained as the contraction 
$\RRR_{ik}=g^{jl}\RRR_{ijkl}$ and $\RRR=g^{ik}\RRR_{ik}$ will 
denote the scalar curvature.

The so--called Weyl tensor is then 
defined by the decomposition formula (see~\cite[Chapter~3,
Section~K]{gahula}) of the Riemann tensor in dimension $n\geq 3$,
\begin{equation}\label{weyl}
\WWW_{ijkl}=\RRR_{ijkl}+\frac{\RRR}{(n-1)(n-2)}(g_{ik}g_{jl}-g_{il}g_{jk})
- \frac{1}{n-2}(\RRR_{ik}g_{jl}-\RRR_{il}g_{jk}
+\RRR_{jl}g_{ik}-\RRR_{jk}g_{il})\,.
\end{equation}
The tensor $\WWW$ satisfies all the symmetries of the curvature tensor
and all its traces with the metric are zero, 
as it can be easily seen from the above formula.\\
In dimension three, $\WWW$ is identically zero for every Riemannian
manifold $(M,g)$, and it becomes relevant instead when $n\geq 4$ since
it vanishes if and only if $(M,g)$ is locally conformally flat. This latter condition means that around every point $p\in M$ there is a conformal
deformation $\widetilde{g}_{ij}=e^fg_{ij}$ of the original metric $g$,
such that the new metric is flat, 
namely, the Riemann tensor associated to
$\widetilde{g}$ is zero in $U_p$ (here $f:U_p\to\RR$ is a smooth function defined in a open
neighborhood $U_p$ of $p$).

\section{Short time existence}\label{short}

\begin{teo}\label{short time} Let $\rho<1/2(n-1)$. 
Then, the evolution equation~\eqref{eqflow} has a unique solution for a positive
time interval on any smooth, $n$--dimensional, compact Riemannian manifold $M$ (without boundary) 
for any initial metric $g_0$.
\end{teo}

\begin{proof}
We first compute the linearized operator $DL_{g_0}$ of the operator
$L=-2(\Ric-\rho\RRR g)$ at a metric $g_0$.
The Ricci tensor and the scalar curvature have the following
linearizations, see~\cite[Theorem~1.174]{besse} or~\cite{topping1},
where we use the metric $g_0$ to lower and raise indices and to take traces,
\begin{eqnarray*}
D\Ric_{g_0}(h)_{ik}&=& \frac{1}{2}\Big(-\Delta
h_{ik}-\nabla_i\nabla_k\tr(h)
+\nabla_i\nabla^th_{tk}+\nabla_k\nabla^th_{it}\Big)+{\mathrm {LOT}},\\
D\RRR_{g_0}(h)&=& -\Delta(\tr h)+\nabla^s\nabla^t h_{st}+{\mathrm {LOT}},
\end{eqnarray*}
here ${\mathrm {LOT}}$ stands for {\em lower order terms}.\\
Then, the linearization of $L$ at $g_0$ is given by 
\begin{eqnarray*}
DL_{g_0}(h)_{ik}&=&-2\big(D\Ric_{g_0}(h)_{ik}-\rho
D\RRR_{g_0}(h)(g_0)_{ik}\big)+2\rho\RRR_{g_0}h_{ik}\\
&=& \Delta
h_{ik}+\nabla_i\nabla_k\tr(h)-\nabla_i\nabla^th_{tk}-\nabla_k\nabla^th_{it}
-2\rho\big(\Delta(\tr h)-\nabla^s\nabla^t h_{st}\big)(g_{0})_{ik}+{\mathrm {LOT}},
\end{eqnarray*}
for every bilinear form $h\in\Gamma(S^2M)$.
Now, we obtain the principal symbol of the linearized operator in the direction
of an arbitrary cotangent vector $\xi$ by replacing each covariant derivative
$\nabla_{\alpha}$, appearing in the higher order terms, with the corresponding component $\xi_{\alpha}$, 
\begin{eqnarray*}
\sigma_{\xi}(DL_{g_0})(h)_{ik} &=&
\xi^t\xi_th_{ik}+\xi_i\xi_k\tr_{g_0}(h)-\xi_i\xi^th_{kt}-\xi_k\xi^th_{it}\\
&&-2\rho\xi^t\xi_t\tr_{g_0}(h)(g_0)_{ik}+2\rho\xi^t\xi^sh_{ts}(g_0)_{ik}\,.
\end{eqnarray*}
As usual, since the symbol is homogeneous we can assume that $|\xi|_{g_0}=1$ and we perform all the computations in an 
orthonormal basis $\{e_i\}_{i=1,\dots, n}$ of $T_pM$ such that
$\xi=g_0(e_1,\cdot)$, that is $\xi_i=0$ for $i\neq 1$.\\
Hence, we obtain,
\begin{eqnarray*}
\sigma_{\xi}(DL_{g_0})(h)_{ik} &=&
h_{ik}+\delta_{i1}\delta_{k1}\tr_{g_0}(h)-\delta_{i1}h_{k1}-\delta_{k1}h_{i1}
-2\rho\tr_{g_0}(h)\delta_{ik}+2\rho h_{11}\delta_{ik}.
\end{eqnarray*}
that can be represented, in the coordinates system
\begin{equation*}
(h_{11}, h_{22}, \dots, h_{nn}, h_{12}, \dots, h_{1n}, h_{23}, h_{24}, \dots, h_{n-1,n})
\end{equation*}
for any $h\in\Gamma(S^2M)$, by the following matrix

\begin{equation*}
\sigma_{\xi}(DL_{g_0})=
\left(\begin{array}{c|c|c} 
\begin{array}{cccc}
0 & 1-2\rho &\dots & 1-2\rho  \\
\vdots &  & A[n-1] & \\
0 &  &\phantom{ A[n-1]}   & 
\end{array}& 0 & \phantom{aaa}0\phantom{aaa} \\
  \hline
\phantom{aaa} & \phantom{aaa} & \phantom{aaa} \\
0 & 0 & 0\\
\phantom{aaa} & \phantom{aaa} & \phantom{aaa} \\
\hline
\phantom{aaa} & \phantom{aaa} & \phantom{aaa} \\
0 & 0 & {\text{\Large{${\mathrm {Id}_{(n-1)(n-2)/2}\,\,\,\,}$}}}\\
\phantom{aaa} & \phantom{aaa} & \phantom{aaa}
\end{array}\right)
\end{equation*}
where $A[n-1]$ is the $(n-1)\times(n-1)$ matrix given by
\begin{equation*}
A[n-1]=\left(
\begin{array}{cccc}
 1-2\rho & -2\rho & \dots & -2\rho \\
 -2\rho & 1-2\rho &\dots & -2\rho\\
 \vdots & \vdots &\ddots&\vdots\\
-2\rho & -2\rho &\dots& 1-2\rho
\end{array}
\right)\,.
\end{equation*}
We can see that there are at least $n$ null eigenvalues, as it
should be expected by the diffeomorphisms invariance of the operator
$L$, and $(n-1)(n-2)/2$ eigenvalues equal to $1$. The remaining $n-1$
eigenvalues can be computed by the following lemma which is easily
proved by induction on the dimension of $A$.
\begin{lemma}\label{det_A}
Let $A[m]$ be the $m\times m$ matrix
\begin{equation}\label{matrix_A}
A[m]=\left(
\begin{array}{cccc}
 1-2\rho & -2\rho & \dots & -2\rho \\
 -2\rho & 1-2\rho &\dots & -2\rho\\
 \vdots & \vdots &\ddots&\vdots\\
-2\rho & -2\rho &\dots& 1-2\rho
\end{array}
\right)\,.
\end{equation}
Then, there holds
\begin{equation*}
\det(A[m]-\lambda \mathrm{Id}_m)=(1-\lambda)^{(m-1)}(1-2m\rho-\lambda)\,.
\end{equation*}
\end{lemma}

Using this lemma, we conclude that the eigenvalues of the principal
symbol of $DL_{g_0}$ are $0$ with multiplicity $n$, $1$ with
multiplicity $\frac{(n+1)(n-2)}{2}$ and $1-2(n-1)\rho$ with multiplicity
$1$.

Now we apply the so--called {\em DeTurck's
  trick}~\cite{deturck,deturck2} 
to show that the RB flow is equivalent to a Cauchy problem for a strictly 
parabolic operator, modulo the action of the diffeomorphism group of
$M$. Let $V:\Gamma(S^2M)\to\Gamma(TM)$ be ``DeTurck's'' vector field defined by 
\begin{equation}\label{eq_V}
V^j(g)=-g_0^{jk}g^{pq}\nabla_p\Big(\frac{1}{2}\tr_g(g_0)g_{qk}-(g_0)_{qk}\Big)
=-\frac{1}{2}g_0^{jk}g^{pq}\big(\nabla_k(g_0)_{pq}-\nabla_p(g_0)_{qk}-\nabla_q(g_0)_{pk}\big)\,,
\end{equation}
where $g_0$ is a fixed Riemannian metric on $M$ and $g_0^{jk}$ are the components of the inverse matrix of $g_0$.\\
The DeTurck's trick (see~\cite{deturck,deturck2} for details) states
that in order to show the smooth existence part of the theorem, we
only need to check that the operator $D(L-\mathscr{L}_V)_{g_0}$ is strongly
elliptic, where $\mathscr{L}_V$ is the Lie derivative operator in the
direction of $V$.\\
The principal symbol of this latter operator, with the same notations 
used above, is well known and is given by
\begin{equation}
\sigma_{\xi}(D\mathscr{L}_V)_{g_0}(h)_{ik} =
\delta_{i1}\delta_{k1}\tr_{g_0}(h)-\delta_{i1}h_{k1}-\delta_{k1}h_{i1}.\nonumber
\end{equation}
Then, we can easily see that the linearized DeTurck--Ricci--Bourguignon
operator has principal symbol in the  direction $\xi$, with respect to
an orthonormal basis $\{\xi^{\flat}, e_2,\dots,e_n\}$, given by
\begin{equation*}
\sigma_{\xi}((D(L-\mathscr{L}_V)_{g_0})=
\left(\begin{array}{c|c|c} 
\begin{array}{cccc}
1 & -2\rho &\dots & -2\rho  \\
\vdots &  & A[n-1] & \\
0 &  &   & 
\end{array}& 0 & \phantom{aaa}0\phantom{aaa} \\
  \hline
\phantom{aaa} & \phantom{aaa} & \phantom{aaa} \\
0 & {\text{\Large{${\mathrm {Id}_{(n-1)}\,\,\,\,}$}}} & 0\\
\phantom{aaa} & \phantom{aaa} & \phantom{aaa} \\
\hline
\phantom{aaa} & \phantom{aaa} & \phantom{aaa} \\
0 & 0 & {\text{\Large{${\mathrm {Id}_{(n-1)(n-2)/2}\,\,\,\,}$}}}\\
\phantom{aaa} & \phantom{aaa} & \phantom{aaa}
\end{array}\right)\,,
\end{equation*}
expressed in the coordinates system
\begin{equation*}
(h_{11},  h_{22}, \dots, h_{nn}, h_{12}, h_{13}, \dots, h_{1n}, h_{23}, h_{24}, \dots, h_{n-1,n})
\end{equation*}
for any $h\in\Gamma(S^2M)$.\\
Using Lemma~\ref{det_A} again, this matrix has $\frac{n(n+1)}{2}-1$ eigenvalues
equals to $1$ and $1$ eigenvalue equal to $1-2(n-1)\rho$. Therefore,
by the DeTurck's trick, a sufficient condition for the existence of a solution is that 
$\rho<\frac{1}{2(n-1)}$.

 The uniqueness part of the theorem is proven in the same way as for the Ricci flow (see~\cite{hamilton9}). The RB flow is equivalent, via the one parameter group of diffeomorphisms generated by DeTurck's vector field, to the DeTurck--RB flow which is strictly parabolic. On the other hand, the one parameter group of diffeomorphisms satisfies the
 harmonic map flow introduced by Eells and Sampson in~\cite{eelsam}, which is also parabolic. These two facts allow to state the uniqueness of the solution for the RB flow (see~\cite[Chapter~3, Section~4]{chknopf} for more details). 
\end{proof}

\section{Evolution of the curvature}

\subsection{The evolution of curvature}

As the metric tensor evolves as 
\begin{equation*}
\dt g_{ij} \,=\, -2(\RRR_{ij} -\rho \RRR g_{ij}) \,,
\end{equation*}
it is easy to see, differentiating the identity
$g_{ij}g^{jl}=\delta_i^l$, that
\begin{equation}\label{eqflow-1}
\dt g^{jl} \,=\, 2(\Ric^{jl} -\rho \RRR g^{jl}) \,.
\end{equation}
It follows that the canonical volume measure $\mu$ satisfies
\begin{equation*}
\frac{d\mu}{dt}=\dt \sqrt{\det g_{ij}}\,\mathscr{L}^n=
\frac{\sqrt{\det g_{ij}}
    g^{ij}\dt g_{ij}}{2}\,\mathscr{L}^n=(n\rho-1)\RRR\sqrt{\det g_{ij}}\,\mathscr{L}^n=
(n\rho-1)\RRR \,\mu\,.
\end{equation*}

Computing in a normal coordinates system, the evolution equation for the
Christoffel symbols is given by
\begin{align*}
\frac{\partial\,}{\partial t}\Gamma_{jk}^i
=&\,\frac{1}{2}g^{il}\left\{
\frac{\partial\,}{\partial x_j}\left(\frac{\partial\,}{\partial t}g_{kl}\right) + 
\frac{\partial\,}{\partial x_k}\left(\frac{\partial\,}{\partial t}g_{jl}\right) - 
\frac{\partial\,}{\partial x_l}\left(\frac{\partial\,}{\partial
    t}g_{jk}\right)\right\}\\
&\,+\frac{1}{2}\frac{\partial\,}{\partial t}g^{il}\left\{
\frac{\partial\,}{\partial x_j}g_{kl} + 
\frac{\partial\,}{\partial x_k}g_{jl} - 
\frac{\partial\,}{\partial x_l}g_{jk}\right\}\\
=&\,\frac{1}{2}g^{il}\left\{
\nabla_j\left(\frac{\partial\,}{\partial t}g_{kl}\right) + 
\nabla_k\left(\frac{\partial\,}{\partial t}g_{jl}\right) - 
\nabla_l\left(\frac{\partial\,}{\partial t}g_{jk}\right)\right\}\\
=&\,-g^{il}\left\{
\nabla_j(\RRR_{kl}-\rho\RRR g_{kl}) + 
\nabla_k(\RRR_{jl}-\rho\RRR g_{jl}) - 
\nabla_l(\RRR_{jk}-\rho\RRR g_{jk})\right\}\\
=&\,
-\nabla_j\RRR_k^i
-\nabla_k\RRR_j^i
-\nabla^i\RRR_{jk}
+\rho(
\nabla_j\RRR\delta^i_k + 
\nabla_k\RRR\delta_j^i +
\nabla^i\RRR g_{jk})\,.
\end{align*}

\begin{prop}\label{eqrmprop} Along the RB flow on a $n$--dimensional Riemannian manifold $(M,g)$, the curvature tensor, the Ricci tensor and the scalar curvature satisfy the following evolution equations
\begin{eqnarray}\label{eqrm}
\dt \RRR_{ijkl} &=& \Delta \RRR_{ijkl} + 2(\BBB_{ijkl}-\BBB_{ijlk}-\BBB_{iljk}+\BBB_{ikjl}) \\ \nonumber
&&-\,g^{pq}\big(\RRR_{pjkl}\RRR_{qi}+\RRR_{ipkl}\RRR_{qj}+\RRR_{ijpl}\RRR_{qk}+\RRR_{ijkp}\RRR_{ql}\big)\\\nonumber
&&-\,\rho\big(\nabla_{i}\nabla_{k}\RRR\,g_{jl}-\nabla_{i}\nabla_{l}\RRR\,g_{jk}-\nabla_{j}\nabla_{k}\RRR\,g_{il}+\nabla_{j}\nabla_{l}\RRR\,g_{ik}\big)\\\nonumber
&&+\,2\rho\RRR\, \RRR_{ijkl}\,,
\end{eqnarray}
where the tensor $\BBB$ is defined as $\BBB_{ijkl}=g^{pq}g^{rs}\RRR_{ipjr}\RRR_{kqls}$.
\begin{eqnarray}\label{eqrc}
\dt \RRR_{ik} &=& \Delta \RRR_{ik} + 2g^{pq}g^{rs}\RRR_{pirk}\RRR_{qs}-2g^{pq}\RRR_{pi}\RRR_{qk}\\\nonumber
&&-(n-2)\rho \nabla_{i}\nabla_{k}\RRR -\rho \Delta\RRR\, g_{ik}\,,
\end{eqnarray}
\begin{eqnarray}\label{eqsc}
\dt\RRR &=& \big(1-2(n-1)\rho\big)\Delta\RRR +2|\Ric|^{2}-2\rho\RRR^{2}\,.
\end{eqnarray}
\end{prop}

\begin{proof}
The following computation is analogous to the one for the Ricci flow
performed by Hamilton~\cite{hamilton1}.\\
By the first variation formula for the $(4,0)$--Riemann tensor
(see~\cite[Theorem~1.174]{besse} or~\cite{topping1}), we have in general
\begin{align*}
\dt\Rm(X,Y,W,&Z)=\frac{1}{2}\big(h(\mathrm{Riem}(X,Y)W,Z)-h(\mathrm{Riem}(X,Y)Z,W)\big)\\
& -\frac{1}{2}\big(-\nabla^2_{Y,W}h(X,Z)-\nabla^2_{X,Z}h(Y,W)+\nabla^2_{X,W}h(Y,Z)+\nabla^2_{Y,Z}h(X,W)\big)\,,
\end{align*}
where $X,Y,W,Z\in\Gamma(TM)$ are vector fields and $h=\dt g$.\\
Along the RB flow $h=-2(\Ric-\rho\RRR g)$, therefore
\begin{align*}
\dt\Rm(X&,Y,W,Z)=-\Ric(\mathrm{Riem}(X,Y)W,Z)+\Ric(\mathrm{Riem}(X,Y)Z,W)\\
&-\nabla^2_{Y,W}\Ric(X,Z)-\nabla^2_{X,Z}\Ric(Y,W)+\nabla^2_{X,W}\Ric(Y,Z)+\nabla^2_{Y,Z}\Ric(X,W)\\
&-\rho\big(-\nabla^2_{Y,W}\RRR g(X,Z)-\nabla^2_{X,Z}\RRR g(Y,W)+\nabla^2_{X,W}\RRR g(Y,Z)+\nabla^2_{Y,Z}\RRR g(X,W)\big)\\
&+2\rho\RRR\Rm(X,Y,W,Z)\,.
\end{align*}
Using the second Bianchi identity and the commutation formula for
second covariant derivatives, we obtain the following equation for the Laplacian of the
Riemann tensor,
\begin{eqnarray*}
\Delta \Rm(X,Y,W,Z)&=& -\nabla^2_{Y,W}\Ric(X,Z)-\nabla^2_{X,Z}\Ric(Y,W)\\
&&+\nabla^2_{X,W}\Ric(Y,Z)+\nabla^2_{Y,Z}\Ric(X,W)\\
&& -\Ric(\mathrm{Riem}(W,Z)Y,X)+\Ric(\mathrm{Riem}(W,Z)X,Y)\\
&& -2\big(\BBB(X,Y,W,Z)-\BBB(X,Y,Z,W)\\
&& +\BBB(X,W,Y,Z)-\BBB(X,Z,Y,W)\big)\,.
\end{eqnarray*}
Plugging it in the evolution equation, we obtain
\begin{eqnarray*}
\dt\Rm(X,Y,W,Z)&=& \Delta\Rm(X,Y,W,Z)-\rho\big(\nabla^2\RRR\varowedge g)(X,Y,W,Z)\\
&& +2\big(\BBB(X,Y,W,Z)-\BBB(X,Y,Z,W)\\
&& +\BBB(X,W,Y,Z)-\BBB(X,Z,Y,W)\big)\\
&& -\Ric(\mathrm{Riem}(X,Y)W,Z)+\Ric(\mathrm{Riem}(X,Y)Z,W)\\
&&-\Ric(\mathrm{Riem}(W,Z)X,Y)+\Ric(\mathrm{Riem}(W,Z)Y,X)\\
&&+2\rho\RRR\Rm(X,Y,W,Z)\,,
\end{eqnarray*}
which is formula~\eqref{eqrm} once written in coordinates. Here 
the symbol $\varowedge$ denotes the Kulkarni--Nomizu product of
two symmetric bilinear forms $p$ and $q$, defined by
\begin{equation*}                                                               
(p\varowedge q)(X,Y,Z,T)=
p(X,Z)q(Y,T)+p(Y,T)q(X,Z)-p(X,T)q(Y,Z)-p(Y,Z)q(X,T)\,,
\end{equation*}
for every tangent vectors fields $X,Y,Z,T\in\Gamma(TM)$.

Taking into account the evolution equation for the inverse of the
metric~\eqref{eqflow-1},  contracting equation~\eqref{eqrm} and using again the second Bianchi
identity, formula~\eqref{eqrc} follows (see~\cite{hamilton1} for details).
Contracting again one gets the evolution equation~\eqref{eqsc} for the
scalar curvature.
\end{proof}

\subsection{The Uhlenbeck's trick and the evolution of the curvature operator}\label{uhlenbeck}\ \\
In this subsection we want to study the evolution equation of the
curvature operator, as it was done for the Ricci flow by
Hamilton in~\cite{hamilton2}.\\
First of all, we simplify the expression of the reaction term in
equation~\eqref{eqrm} by means of the so called Uhlenbeck's trick~\cite{hamilton2}. Briefly, we will 
relate the curvature tensor of the evolving metric to an evolving
tensor of an abstract bundle with the same symmetries of the curvature
(see Proposition~\ref{proppm}) and a nicer evolution equation;
afterwards we will find a suitable orthonormal moving frame of $(TM,
g(t))$ and write the evolution equation of the coordinates of the
Riemann tensor with respect to this frame. The result will be 
a system of {\em scalar} evolution equations and no more a
tensorial equation, (see~\cite{chknopf} for more details on this method in the case of Ricci flow).\\

Let $\big(M, g(t)\big)_{t\in[0,T)}$ be the solution of the RB flow with initial data $g_0$ and consider on the tangent bundle $TM$ the family of endomorphisms  $\big\{\varphi(t)\big\}_{t\in [0, T)}$ defined by the following evolution equation 
\begin{equation}\label{evol_isom}
\left\{
\begin{array}{l}
 \Derpar{t}  \varphi(t)=\Ric_{g(t)}^{\#}\circ\varphi(t) -\rho\RRR_{g(t)}\varphi(t)\,, \\
\varphi(0)=\Id_{TM}\,,
\end{array}
\right.
\end{equation}
where $\Ric_{g(t)}^{\#}$ is the endomorphism of the tangent bundle
canonically associated to the Ricci tensor by raising an index.

For every point $p$ of the manifold $M$, the evolution equation~\eqref{evol_isom}
represents a system of linear ODEs on the fiber $T_pM$, therefore a
unique solution exists as long as the RB flow exists. Moreover, if we
let $(h(t))_{t\in[0,T)}$ be the family of bundle metrics defined by
$h(t)=\varphi(t)^*(g(t))$, where $\varphi(t)$ satisfies 
system~\eqref{evol_isom}, then $h(t)=g_0$ for every $t\in[0,T)$. As 
$$
\forall t\in[0,T), \qquad \varphi(t): (TM, g_0)\to(TM, g(t))
$$
is a bundle isometry, the pull--back via $\varphi(t)$ of the
Levi--Civita connection associated to $g(t)$ 
is a connection on $TM$ compatible with the metric $g_0$. In the
following, we will denote by $(V,h)$ the vector bundle $(TM,g_0)$ 
in order to stress out the fact that we are not considering the Levi--Civita 
connection associated to $g_0$, but the family of
time--dependent connections $D(t)$ defined via the bundle isometries
$\varphi(t)$.

The following lemma states some basic properties of these pull--back
connections.

\begin{lemma}\label{pullbackD}(see~\cite[Chapter~6, Section~2]{chknopf}
Let $D(t):\Gamma(TM)\times\Gamma(V)\to\Gamma(V)$ be the pull--back connection defined by 
\begin{equation*}
D(t)_X\zeta=\varphi(t)^*\big(\nabla^{g(t)}_X\big(\varphi(t)(\zeta)\big)\big)\,,
\end{equation*}
$\forall  t\in[0,T)\,, \forall X\in\Gamma(TM),\forall \zeta\in\Gamma(V)$, where $\nabla^{g(t)}$ is the Levi--Civita connection of $g(t)$.\\
Let again $D(t)$ be the canonical extension to the tensor powers of
$V$ and $T$ be a covariant tensor on $M$. 
Then, for every $t\in[0,T)$ and $X\in\Gamma(TM)$ there holds
$$
D(t)_X\big(\varphi(t)^* (T)\big)=\varphi(t)^*\big(\nabla^{g(t)}_XT\big)\,. 
$$
In particular, $D(t)_Xh=\varphi^*(\nabla^{g(t)}_Xg(t))=0$, so every connection of the family $D(t)$ is compatible with the bundle metric $h$ on $V$.\\
Let $D^2:\Gamma(TM)\times\Gamma(TM)\times\Gamma(V)\to\Gamma(V)$ be the second covariant derivative defined by
$$
D^2_{X,Y}(\zeta)=D_X(D_Y\zeta)-D_{\nabla^{g(t)}_XY}\zeta\,,\qquad\forall X,Y\in \Gamma(TM)\,,\forall \zeta\in \Gamma(V)\,,
$$
and the rough Laplacian defined by $\Delta_D=\tr_g(D^2)$.
Then, for every covariant tensor $T$ on $M$, there hold
\begin{eqnarray}
D^2_{X,Y}\big(\varphi^*(T)\big)&=&\varphi^*(\nabla^2_{X,Y}T)
\qquad\forall X,Y\in \Gamma(TM)\,,\\
\Delta_D\big(\varphi^*(T)\big)&=&\varphi^*(\Delta_gT)\,.
\end{eqnarray}
\end{lemma}

\begin{rem}\label{remriem} 
Let $\M{R}\in End(\Lambda^2M)$ be the \emph{Riemann curvature
  operator} defined by
\begin{equation}
\langle\M{R}(X\wedge Y),W\wedge Z\rangle=\Rm(X,Y,W,Z)\,,
\end{equation}
where $\langle\,,\,\rangle$ is the linear extension of $g$ to the
exterior powers of $TM$.\\ In the following, we use a convention on
the Lie algebra structure of $\Lambda^2M$ which is different from the
original one chosen by Hamilton in~\cite{hamilton2}.
More precisely, with his convention the
eigenvalues of the curvature operator are twice the sectional
curvatures, whereas with our convention the curvature operator has the
sectional curvatures as eigenvalues. In particular, every formula differs from the
corresponding one in the usual theory of the Ricci flow by a factor $2$ (see
also~\cite[Chapter~6, Section~3]{chknopf} for the
details). We recall that $\M{R}$ can be considered as an element of
$\Gamma(S^2(\Lambda^2M))$, and the following equations hold true
\begin{eqnarray*}
\RRR &=&2\sum_{i<k}\M{R}_{(ik)}^{(ik)}\,;\\
(\M{R}^2)_{ijkl}&=&\BBB_{ijkl}-\BBB_{ijlk}\,;\\
(\M{R}\#\M{R})_{ijkl}&=& \BBB_{ikjl}-\BBB_{iljk}\,.
\end{eqnarray*}
where $\BBB$ is defined as in Proposition~\ref{eqrmprop}. For more details on the structure of the curvature operator we refer again the reader to~\cite[Chapter~6, Section~3]{chknopf}.
\end{rem}

We now consider the pull--back of the Riemann curvature tensor and the curvature operator.

\begin{prop}\label{proppm}
Let $\Pm$ be the pull--back of the Riemann curvature tensor via the family of bundle isometries $\{\varphi(t)\}_{t\in[0,T)}$.
The following statements hold true:
\begin{enumerate}
\item $\Pm$ has the same symmetry properties as $\Rm$, i.e. it can be
  seen as an element of $\Gamma(S^2(\Lambda^2V))$ and it satisfies the first Bianchi identity;
\item For every $p\in M$ and $t\in[0,T)$ the {\em algebraic curvature
    operator} $\M{P}(p,t)\in End(\Lambda^2V_p)$ (see
  Remark~\ref{P_alg}), defined by $\varphi\circ\M{P}=\M{R}\circ\varphi$ 
    has the same eigenvalues as
  $\M{R}(p,t)$. In particular, $\M{P}$ is positive (nonnegative)
  definite if and only if $\M{R}$ is positive (nonnegative) definite;
\item $\Pic(t)=\tr_h\big(\Pm(t)\big)=\varphi(t)^*(\Ric_{g(t)})$;
\item $\PPP=\tr_h(\Pic(t))=\RRR_{g(t)}$;
\item $\BBB(\Pm)=\varphi^*\big(\BBB(\Rm)\big)$,
where $\BBB$ is defined in the same way as in Proposition~\ref{eqrmprop} for a generic element of $S^2(\Lambda^2V^*)$.
\end{enumerate}
\end{prop}

Finally, we can compute the evolution of $\Pm$ and $\M{P}$.
\begin{prop}\label{uhlen_reward}
The tensors $\Pm$ and $\M{P}$ satisfy respectively the following evolution equations
\begin{eqnarray}\label{evolDU}
\Derpar{t}(\Pm)_{abcd}&=& \Delta_D(\Pm)_{abcd}-\rho(\varphi^*(\nabla^2\RRR)\varowedge h)_{abcd}\\
&& + 2\big(\BBB(\Pm)_{abcd}-\BBB(\Pm)_{abdc}+\BBB(\Pm)_{acbd}-\BBB(\Pm)_{adbc}\big)\nonumber\\
&&-2\rho\PPP\,\Pm_{abcd}\,,\nonumber\\
\dt\M{P}&=&\Delta_{D}\M{P}-2\rho\varphi^*(\nabla^2 \tr_h(\M{P}))\varowedge h+2\M{P}^2+2\M{P}^{\#}-4\rho \tr_h(\M{P})\M{P},\label{evol_P}
\end{eqnarray}
where $\tr_h(\M{P}(t))=\tr_{g(t)}(\M{R}(t))=\frac{1}{2}\RRR(t)$. 
\end{prop}
\begin{rem}
On the right hand side of equation~\eqref{evolDU} the term
$\varphi^*(\nabla^2 \RRR)$ there appears (i.e. the pull--back of the Hessian of the
scalar curvature, seen as a symmetric 2--form on the tangent bundle) 
and it cannot be expressed in terms of the connection $D(t)$.
\end{rem}
\begin{proof}
Let $\zeta_1,\dots,\zeta_4$ be sections of $V$; then combining the
evolution equations of the Riemann tensor~\eqref{eqrm} 
and of the bundle isometry $\varphi$~\eqref{evol_isom}, we obtain
\begin{eqnarray*}
&&\Derpar{t}(\Pm)(\zeta_1,\zeta_{2},\zeta_{3},\zeta_4)= \\
&& \varphi^*\Big(\Derpar{t}\Rm\Big)(\zeta_1,\zeta_{2},\zeta_{3},\zeta_4)+\Rm\Big(\derpar{\varphi}{t}(\zeta_1),\varphi(\zeta_2),\varphi(\zeta_3),\varphi(\zeta_4)\Big)\\
&& +\Rm\Big(\varphi(\zeta_1),\derpar{\varphi}{t}(\zeta_2),\varphi(\zeta_3),\varphi(\zeta_4)\Big)
+\Rm\Big(\varphi(\zeta_1),\varphi(\zeta_2),\derpar{\varphi}{t}(\zeta_3),\varphi(\zeta_4)\Big)\\
&&+\Rm\Big(\varphi(\zeta_1),\varphi(\zeta_2),\varphi(\zeta_3),\derpar{\varphi}{t}(\zeta_4)\Big)
\\
&=& \varphi^*(\Delta_g\Rm)(\zeta_1,\zeta_{2},\zeta_{3},\zeta_4)-\rho\varphi^*(\nabla^2\RRR\varowedge g)(\zeta_1,\zeta_{2},\zeta_{3},\zeta_4)\\
&& + 2\varphi^*\big(\BBB(\Rm)(\zeta_1,\zeta_{2},\zeta_{3},\zeta_4)-\BBB(\Rm)(\zeta_1,\zeta_{2},\zeta_{4},\zeta_3)-\BBB(\Rm)(\zeta_1,\zeta_{4},\zeta_{2},\zeta_3)\big.\\
&&\big.+\BBB(\Rm)(\zeta_1,\zeta_{3},\zeta_{2},\zeta_4) \big)+2\rho\RRR\varphi^*(\Rm)(\zeta_1,\zeta_{2},\zeta_{3},\zeta_4)\\
&& - \Rm\Big(\Ric^{\#}\circ\varphi(\zeta_1),\varphi(\zeta_2),\varphi(\zeta_3),\varphi(\zeta_4)\Big)- \Rm\Big(\varphi(\zeta_1),\Ric^{\#}\circ\varphi(\zeta_2),\varphi(\zeta_3),\varphi(\zeta_4)\Big)\\
&&- \Rm\Big(\varphi(\zeta_1),\varphi(\zeta_2),\Ric^{\#}\circ\varphi(\zeta_3),\varphi(\zeta_4)\Big)- \Rm\Big(\varphi(\zeta_1),\varphi(\zeta_2),\varphi(\zeta_3),\Ric^{\#}\circ\varphi(\zeta_4)\Big)\\
&& + \Rm\Big(\big(\Ric^{\#}\circ\varphi-\rho\RRR\varphi\big)(\zeta_1),\varphi(\zeta_2),\varphi(\zeta_3),\varphi(\zeta_4)\Big)\\
&&+\Rm\Big(\varphi(\zeta_1),\big(\Ric^{\#}\circ\varphi-\rho\RRR\varphi\big)(\zeta_2),\varphi(\zeta_3),\varphi(\zeta_4)\Big)\\
&&+\Rm\Big(\varphi(\zeta_1),\varphi(\zeta_2),\big(\Ric^{\#}\circ\varphi-\rho\RRR\varphi\big)(\zeta_3),\varphi(\zeta_4)\Big)\\
&&+\Rm\Big(\varphi(\zeta_1),\varphi(\zeta_2),\varphi(\zeta_3),\big(\Ric^{\#}\circ\varphi-\rho\RRR\varphi\big)(\zeta_4)\Big)\\
&=& \Delta_D(\Pm)(\zeta_1,\zeta_{2},\zeta_{3},\zeta_4)-\rho(\varphi^*(\nabla^2\RRR)\varowedge h)(\zeta_1,\zeta_{2},\zeta_{3},\zeta_4)\\
&& + 2\varphi^*\big(\BBB(\Pm)(\zeta_1,\zeta_{2},\zeta_{3},\zeta_4)-\BBB(\Pm)(\zeta_1,\zeta_{2},\zeta_{4},\zeta_3)-\BBB(\Pm)(\zeta_1,\zeta_{4},\zeta_{2},\zeta_3)\big.\\
&&\big.+\BBB(\Pm)(\zeta_1,\zeta_{3},\zeta_{2},\zeta_4) \big)-2\rho \PPP \Pm(\zeta_1,\zeta_{2},\zeta_{3},\zeta_4)\,,
\end{eqnarray*}
where we used several identities stated above.
For $\zeta_1,\dots,\zeta_4$ belonging to a local frame we get the desired equation~\eqref{evolDU}.\\
Combining the evolution equation for $\Pm$ with the formulas
stated in Remark~\ref{remriem}, we find the evolution equation of
$\M{P}$.
\end{proof}

\begin{rem}\label{P_alg}
It must be noticed that, even though for every $p\in M$ and $t\in[0,T)$, the tensor $\M{P}(p,t)$ belongs to the set of algebraic curvature operators $\M{C}_b(V_p)$, in general it does not coincide with the curvature operator of the pull--back connection $D(t)$. In the present literature the pull--back tensor is
always denoted by $\Rm$ and this abuse of notation is somehow
misleading, suggesting the wrong impression that
$\Pm(t)=\varphi(t)^*(\Rm_{g(t)})$ is equal to $\Rm_{\varphi(t)^*(g(t))}=\Rm_h$, but
this is not longer true for general isomorphisms of the tangent bundle (however it is true for $\varphi\in Diff(M)$).
\end{rem}

By the Uhlenbeck's trick the evolution equation~\eqref{evol_P} for
$\M{P}$ allows a simpler use of the maximum principle for tensor as
the reaction term is nicer and the metric on $S^2(\Lambda^2V)$ is
independent of time. Moreover, the subsets of $S^2(\Lambda^2V)$
preserved by such PDE correspond to curvature conditions preserved
under the RB flow.

\section{Preserved curvature conditions}

In this section we will use the maximum principle in various
formulations in order to find curvature conditions which are preserved
by the RB flow.

\subsection{The scalar curvature}\ \\ We begin by considering the evolution equation for 
the scalar curvature~\eqref{eqsc}, which behaves as under the Ricci flow.

\begin{prop}\label{Rcresce}
Let $(M, g(t))_{t\in[0,T)}$ be a compact maximal solution of the RB
flow~\eqref{eqflow}. If $\rho\leq\frac{1}{2(n-1)}$, the
minimum of the scalar curvature is nondecreasing along the flow. In
particular if $\RRR_{g(0)}\geq \alpha$, for some $\alpha \in \mathbb{R}$, then $\RRR_{g(t)}\geq \alpha$ for every $t\in[0,T)$. Moreover if $\alpha>0$ then
$T\leq\frac{n}{2(1-n\rho)\alpha}$.
\end{prop}

\begin{proof}
As $\rho\leq\frac{1}{2(n-1)}\leq\frac{1}{n}$, for any
$n> 1$, it follows that
\begin{align*}
\dt\RRR=&\,\big(1-2(n-1)\rho\big)\Delta\RRR
+2|\Ric|^{2}-2\rho\RRR^{2}\\
\geq&\,\big(1-2(n-1)\rho\big)\Delta\RRR
+2\RRR^2/n-2\rho\RRR^{2}\\
\geq&\,\big(1-2(n-1)\rho\big)\Delta\RRR\,,
\end{align*}
hence, by the maximum principle, the minimum of the scalar curvature is
nondecreasing along the RB flow on a compact manifold. In particular,
for every $\alpha\in\R$, the condition $\RRR\geq\alpha$ is preserved.\\
Finally, integrating the inequality
\begin{equation*}
\dt\RRR_{min}\geq2\Bigl(\frac{1}{n}-\rho\Bigr)\RRR_{min}^2\,,
\end{equation*}
that holds almost everywhere for $t\in[0,T)$ (by the Hamilton's trick
(see~\cite{hamilton10},~\cite[Lemma~2.1.3]{Manlib})), we obtain
\begin{equation}\label{bound_sc}
\RRR_{min}(t)\geq\frac{n\alpha}{n-2(1-n\rho)\alpha t}\,,
\end{equation}
that, for $\alpha>0$, gives the estimate on the maximal time of existence.
\end{proof}

\begin{rem}
In the special case of the Schouten flow (when
$\rho=\frac{1}{2(n-1)}$), actually there holds
\begin{equation*}
\dt\RRR \geq \frac{n-2}{n(n-1)}\RRR^2\,,
\end{equation*}
at every point of the manifold, which implies that the scalar
curvature is pointwise nondecreasing and diverges in finite time.
\end{rem}

\begin{rem}
By means of the strong maximum principle, it follows that if the
initial manifold has nonnegative scalar curvature then either the manifold is Einstein ($\Ric=0$) or the scalar curvature becomes positive for every positive time under any RB flow with $\rho\leq \frac{1}{2(n-1)}$.
\end{rem}

\begin{prop}\label{ancient_posR}
Let $(M, g(t))_{t\in(-\infty,0]}$ be a compact, $n$--dimensional, ancient solution of the RB flow~\eqref{eqflow}. If $\rho\leq\frac{1}{2(n-1)}$ then, either $\RRR>0$ or $\Ric \equiv 0$ on $M \times (-\infty, 0]$.
\end{prop}
\begin{proof}
As $g(t)$ is an ancient solution, for every $t_0<t_1\leq0$, we can define
$\widetilde{g}(s)=g(s+t_0)$, which is a solution of the RB flow
for $s\in[0, t_1-t_0]$. Then, we have
$\widetilde{\RRR}_{min}(0)=\RRR_{min}(t_0)$, hence,
from formula~\eqref{bound_sc}
\begin{equation*}
\widetilde{\RRR}_{min}(s)\geq \frac{n}{n\widetilde{\RRR}_{min}^{-1}(0)-2(1-n\rho)s}\,,
\end{equation*}
for every $s\in(0,t_1-t_0]$. In particular, we have
\begin{equation*}
\RRR_{min}(t_1)=\widetilde{\RRR}_{min}(t_1-t_0)\geq \frac{n}{n\RRR_{min}^{-1}(t_0)-2(1-n\rho)(t_1-t_0)}\,.
\end{equation*}
If $\RRR_{min}(t_0)\geq 0$, by Proposition~\ref{Rcresce}, it
follows that $\RRR_{min}(t_1)\geq 0$, so we can assume that
$\RRR_{min}(t_0)<0$, hence
\begin{equation*}
\RRR_{min}(t_1)\geq \frac{n}{n\RRR_{min}^{-1}(t_0)-2(1-n\rho)(t_1-t_0)}>-\frac{n}{2(1-n\rho)(t_1-t_0)}\,,
\end{equation*}
for every $t_1<t_0$, and sending $t_0$ to $-\infty$, we still conclude
that $\RRR_{min}(t_1)\geq 0$. Since this holds for every $t_1\leq
0$  the previous remark implies the result.
\end{proof}

\subsection{Maximum principle for uniformly elliptic operators}\ \\
Let $M$ be a smooth compact manifold, $g(t)$, $t\in[0,T)$, a family of Riemannian metrics on $M$ and $(E, h(t))$ $t\in[0,T)$, be a real vector bundle on $M$, endowed with a (possibly time--dependent) bundle metric. Let $D(t): \Gamma(TM)\times\Gamma(E)\to\Gamma(E)$ be a family of linear connections on $E$ compatible at each time with the bundle metric $h(t)$. We have already seen in Section~\ref{uhlenbeck} how to define the second covariant derivative, using also the Levi-Civita connections $\nabla_{g(t)}$ associated to the Riemannian metrics on $M$.
\begin{dfnz} We consider a second--order linear operator $\M{L}$ on $\Gamma(E)$ that lacks of its 0--th order term, hence it can be written in a local frame field $\{e_i\}_{i=1,\dots,n}$ of $TM$
\begin{equation}\label{op_ELL}
\M{L}=a^{ij}D^2_{e_ie_j}-b^iD_{e_i}
\end{equation}
where $a=a^{ij}e_i\otimes e_j \in\Gamma(S^2(TM))$ is a symmetric $(0,2)$--tensor and $b=b^ie_i$ is a smooth vector field. We say that $\M{L}$ is \emph{uniformly elliptic} if $a$ is uniformly positive definite. 
\end{dfnz}

\begin{rem}
In the previous definition, both the coefficients and the connections are in general time--dependent and we say that $\M{L}$ is uniformly elliptic if it is so for every $t\in[0,T)$ uniformly in time. 
\end{rem}

Weinberger in~\cite{weinberger} proved the maximum principle for systems of solutions of a time--dependent heat equation in the Euclidean space; Hamilton in~\cite{hamilton2} treated the general case of a vector bundle over an evolving Riemannian manifold. 
Here we present a slight generalization of Hamilton's theorem for parabolic equations with uniformly elliptic operator (see~\cite[Theorem~2.2]{savas-smoc} for the "static" version proved by Savas--Halilaj and Smoczyk).\\
As before, $(M,g(t))$ is a smooth compact manifold equipped with a family of Riemannian metrics; we consider a real vector bundle $E$ over $M$, equipped with a fixed bundle metric $h$ and a family of time--dependent connections $D(t)$ compatible at every time with $h$. 
\begin{dfnz}
Let $S\subset E$ be a sub-bundle and denote $S_p=S\cap E_p$ for every $p\in M$. We say that $S$ is \emph{invariant under parallel translation} w.r.t. $D$, if for every curve $\gamma:[0,l]\to M$ and vector $v\in S_{\gamma(0)}$, the unique parallel (w.r.t. $D$) section $v(s)\in E_{\gamma(s)}$ along $\gamma(s)$ with $v(0)=v$ is contained in $S$.
\end{dfnz}
\begin{teo}[Vectorial Maximum Principle]\label{syst_MP}
Let $u:[0,T)\to\Gamma(E)$ be a smooth solution of the following parabolic equation
\begin{equation}\label{vec_PDE}
\Derpar{t}u=\M{L}u+F(u,t)\,,
\end{equation}
where $\M{L}$ is a uniformly elliptic operator as defined in~\eqref{op_ELL} and $F:E\times[0,T)\to E$ is a continuous map, locally Lipschitz in the $E$ factor, which is also fiber--preserving, i.e. $F(v,t)\in E_p$ for every $p\in M$, $v\in E_p$, $t\in[0,T)$.\\
Let $K\subset E$ be a closed sub-bundle (for the metric $h$), invariant under parallel translation w.r.t. $D(t)$, for every $t\in[0,T)$, and convex in the fibers, i.e. $K_p=K\cap E_p$ is convex for every $p\in M$.\\
Suppose that $K$ is preserved by the ODE associated to~\eqref{vec_PDE}, i.e. for every $p\in M$ and $U_0\in K_p$, the solution $U(t)$  of
\begin{equation}
\left\{
\begin{array}{lll}
\frac{dU}{dt}&=&F_p(U(t), t)\,,\\
U(0)&=&U_0\,.
\end{array}
\right.
\end{equation}
remains in $K_p$. Then, if $u$ is contained in $K$ at time $0$,  $u$ remains in $K$, i.e. $u(p,t)\in K_p$ for every $p\in M$, $t\in[0,T)$.
\end{teo}
\begin{proof}(Sketch)
We can follow exactly the detailed proof written in~\cite[Chapter~10,
Section~3]{chowbookII}, provided that we generalize~\cite[Lemma 10.34]{chowbookII} to
the analogue one for uniformly elliptic operator (see again~\cite[Lemma~2.2]{savas-smoc}): if $K\subset E$ satisfies all the
hypothesis of Theorem~\ref{syst_MP} and $u\in\Gamma(E)$ is a smooth section of
$E$, then
\begin{equation*}
  u(p)\in K_p \quad \forall p\in M \quad \Longrightarrow \quad\M{L}(u)_p\in C_{u(p)}K_p \quad \forall p\in M \,,
\end{equation*}
where $C_{u(p)}K_p$ is the tangent cone of the convex set $K_p$ at $u(p)$.
\end{proof}
There is a further generalization of this maximum principle which allows the subset $K$ to be time--dependent.
\begin{teo}[Vectorial Maximum Principle, Time--dependent Set]\label{time_MP}
Let $u:[0,T)\to\Gamma(E)$ be a smooth solution of the parabolic equation~\eqref{vec_PDE}, with the notations of the previous Theorem. For every $t\in[0,T)$, let $K(t)\subset E$ be a closed sub-bundle (for the metric $h$), invariant under parallel translation w.r.t. $D(t)$, convex in the fibers and such that the space--time track
\begin{equation*}
\M{T}=\{(v,t)\in E\times\R: v\in K(t), t\in[0,T)\}
\end{equation*}
is closed in $E\times [0,T)$. Suppose that, for every $t_0\in[0,T)$, $K(t_0)$ is preserved by the ODE associated, i.e. for any $p\in M$, any solution $U(t)$ of the ODE that starts in $K(t_0)_p$ remains in $K(t)_p$, as long as it exists. Then, if $u(0)$ is contained in $K(0)$, $u(p,t)\in K(t)_p$ for ever $p\in M$, $t\in[0,T)$.
\end{teo}
The proof of this theorem, when $K$ depends continuously on time and $F$ does not depend on time is due to Bohm and Wilking (see~\cite[Th. 1.1]{bohmwilk2}). In the general case the proof can be found in~\cite[Chapter 10, Section 5]{chowbookII}, with the usual adaptation to the uniformly elliptic case.

As remarked before, the evolution equation~\eqref{eqrm} of the Riemann tensor has some mixed product of type $\Rm*\Ric$ which makes difficult to understand the behavior of the reaction term. On the other hand, if we perform the Uhlenbeck's trick, the evolution equation~\eqref{evolDU} becomes a little nicer and can be used to understand how the RB flow affects the geometry.\\
More precisely, we use the evolution equation~\eqref{evol_P} for the algebraic curvature operator $\M{P}\in\Gamma(S^2(\Lambda^2V^*))$ to prove that the cone of nonnegative curvature operators is preserved by the RB flow.
\begin{prop}\label{op_nonneg}
Let $(M, g(t))_{t\in[0,T)}$ be a compact solution of the RB flow~\eqref{eqflow} with $\rho<\frac{1}{2(n-1)}$ and such that the initial data $g_0$ has nonnegative curvature operator. Then $\M{R}_{g(t)}\geq 0$ for every $t\in[0,T)$.
\end{prop}
\begin{proof}
We recall the evolution equation~\eqref{evol_P} for $\M{P}=\varphi^{-1}\circ\M{R}\circ\varphi$
\begin{equation*}
\dt\M{P}=\Delta_{D}\M{P}-2\rho\varphi^*(\nabla^2 \tr_h(\M{P}))\varowedge h+2\M{P}^2+2\M{P}^{\#}-4\rho \tr_h(\M{P})\M{P},
\end{equation*}
where $\tr_h(\M{P}(t))=1/2\RRR(t)$ is half of the scalar curvature of the metric $g(t)$. By proposition~\ref{proppm}, it suffices to show that the non negativity of $\M{P}$ is preserved by equation~\eqref{evol_P}. We want to apply the vectorial maximum principle~\ref{syst_MP}, therefore we must show that 
\begin{equation*}
\M{L}(Q)=\Delta_D Q-2\rho\varphi^*(\nabla^2 \tr_h(Q))\varowedge h
\end{equation*}
is a uniformly elliptic operator on the bundle $(\Gamma(S^2(\Lambda^2V^*)),h ,D(t))$. \\
As $\M{L}$ is a linear second order operator, we compute as usual its principal symbol in the arbitrary direction $\xi$. In order to simplify the computations, we choose opportune frames at every point $p\in M$ and time $t\in[0,T)$. Then, let $\{e_i\}_{i=1,\dots, n}$ be an orthonormal basis of $(V_p,h_p)$ such that $\xi=h_p(e_1,\cdot)$. According to the Uhlenbeck's trick (Section~\ref{uhlenbeck}) and the convention on algebraic curvature operators (Section~\ref{uhlenbeck}) we have that $\{f_i=\varphi(t)_p(e_i)\}_{i_1,\dots,n}$ is an orthonormal basis of $T_pM$ with respect to $g(t)_p$, the components of $\varphi(t)_p$ with these choices are $\varphi_i^a=\delta_i^a$ and $\{e_i\wedge e_j\}_{i<j}$ is an orthonormal basis of $\Lambda^2V_p$. Hence, the principal symbol of the operator $\M{L}$ written in these frames is
\begin{eqnarray*}
\sigma_{\xi}(\M{L}Q)_{(ij)(kl)}&=& \xi^p\xi_pQ_{(ij)(kl)}-2\rho\delta_i^a\delta_j^b\delta_k^c\delta_l^d\tr_h(Q)(\xi\otimes\xi\varowedge h)_{(ab)(cd)}\\
&=&|\xi|^2Q_{(ij)(kl)}-2\rho\tr_h(Q)(\xi\otimes\xi\varowedge h)_{(ij)(kl)}\\
&=&Q_{(ij)(kl)}-2\rho\Big(\sum_{p<q}Q_{(pq)(pq)}\Big)\delta_i^1\delta_k^1\delta_{jl}\,,
\end{eqnarray*}
where we used that $|\xi|=1$, $i<j$ and $k<l$ in the last passage. Now it is easy to see that the matrix representing the symbol has the following form
\begin{equation*}
\sigma_{\xi}(\M{L})=
\left(\begin{array}{c|c|c} 
A[n-1]&\begin{array}{ccc}
-2\rho  &\dots & 2\rho  \\
\vdots &  \ddots & \vdots  \\
-2\rho & \dots & -2\rho  
\end{array} & \phantom{aaa}0\phantom{aaa} \\
  \hline
\phantom{aaa} & \phantom{aaa} & \phantom{aaa} \\
0 & {\text{\Large{${\mathrm {Id}_{(n-1)(n-2)/2}\,\,\,\,}$}}} & 0\\
\phantom{aaa} & \phantom{aaa} & \phantom{aaa} \\
\hline
\phantom{aaa} & \phantom{aaa} & \phantom{aaa} \\
0 & 0 & {\text{\Large{${\mathrm {Id}_{N(N-1)/2}\,\,\,\,}$}}}\\
\phantom{aaa} & \phantom{aaa} & \phantom{aaa}
\end{array}\right)\,,
\end{equation*}
where we have ordered the components as follows: first the $n-1$ ones
of the form $(1j)(1j)$ with $j>1$, then the $(n-1)(n-2)/2$ ones of the
form $(ij)(ij)$ with $1<i<j$, and last the $N(N-1)/2$ "non diagonal"
ones, with $N=n(n-1)/2$ and $A$ is the matrix defined
in~\eqref{matrix_A}.\\
By lemma~\ref{det_A} the eigenvalues of the symbol are $1$ with
multiplicity $N(N+1)/2-1$ and $1-2(n-1)\rho$ with multiplicity $1$, since $\rho<1/2(n-1)$ the operator $\M{L}$ is uniformly elliptic.\\
In the second part of the proof we consider the reaction term
$F(Q)=2(Q^2+Q^{\#}-2\rho\tr_h(Q)Q)$. Clearly $F$ is continuous,
locally Lipschitz and fiber--preserving. Let
$\Omega\subset\Gamma(S^2(\Lambda^2V^*))$ be the set of nonnegative
algebraic curvature operators, where we have identified
$S^2(\Lambda^2V^*)\simeq End_{SA}(\Lambda^2V)$ via the metric $h$.
We observe that $\Omega=\{Q:\quad \lambda_N(Q_p)\geq 0\}$, where
$N=n(n-1)/2$ and $\lambda_N$ is the least eigenvalue of $Q_p$. Hence
$\Omega$ is clearly closed, by~\cite[Lemma~10.11]{chowbookII} it is invariant
under parallel translation with respect to every connection $D(t)$ and
it is convex, provided that the function $Q\mapsto\lambda_N(Q_p)$ is
concave. We can rewrite
\begin{equation*}
\lambda_N(Q_p)=\inf_{\{v\in\Lambda^2V_p\,:\, |v|_h=1\}}h(Q_p(v),v)\,;
\end{equation*}
so it is easy to conclude, by the bilinearity of the metric $h$ and
the concavity of $\inf$, that the function defining $\Omega$ is actually
concave and so its superlevels are convex. In order to finish the
proof we have to show that the ODE $dQ/dt=F(Q)$ preserves
$\Omega$. Now,  by standard facts in convex analysis, we only need to
prove that
\begin{equation*}
F_p(Q_p)\in T_{Q_p}\Omega_p\qquad\mbox{for every }p\in M\mbox{ such that }Q_p\in\partial\Omega_p\,,
\end{equation*}
where $\partial\Omega_p=\{Q_p\in\Omega_p:\exists v\in\Lambda^2V_p\quad
Q_p(v,v)=0\}$ and the tangent cone is
\begin{equation*}
T_{Q_p}\Omega_p=\{S_p\in S^2(\Lambda^2V^*_p): S_p(v,v)\geq 0\mbox{ for every }v\in\Lambda^2V_p\mbox{ such that }Q_p(v,v)=0\}
\end{equation*}
Let $v\in\Lambda^2V_p$ and $\{\theta_{\alpha}\}$ be respectively a
null eigenvector of $Q_p$ and an orthonormal basis of $\Lambda^2V_p$
that diagonalizes $Q_p$. Clearly
\begin{equation*}
v=v^{\alpha}\theta_{\alpha}\,,\qquad (Q_p)_{\alpha\beta}=\lambda_{\alpha}\delta_{\alpha\beta}\,.
\end{equation*}
with $\lambda_{\alpha}\geq 0$. Then
\begin{equation*}
(Q_p^2)_{\alpha\beta}=\lambda_{\alpha}^2\delta_{\alpha\beta}\,,\quad
(Q_p^{\#})_{\alpha\beta}=\frac{1}{2}(c_{\alpha}^{\gamma\nu})^2\lambda_{\gamma}\lambda_{\nu}\delta_{\alpha\beta}
\end{equation*}
and
\begin{equation*}
F_p(Q_p)(v,v)=\lambda_{\alpha}^2(v^{\alpha})^2+\frac{1}{2}(c_{\alpha}^{\gamma\nu})^2\lambda_{\gamma}\lambda_{\nu}(v^{\alpha})^2\geq 0\,
\end{equation*}
this completes the proof.
\end{proof}

\subsection{The evolution of the Weyl tensor}\ \\
By means of the evolution equations found for the curvatures, we are also
able to write the equation satisfied by the Weyl tensor along the RB
flow~\eqref{eqflow}. In~\cite{mancat1} the authors compute the
evolution equation of the Weyl tensor during the Ricci flow (see~\cite[Proposition 1.1]{mancat1})
and we use most of their computations here.
\begin{prop}
During the RB flow of an $n$--dimensional Riemannian manifold $(M,g)$
the Weyl tensor satisfies the following evolution equation
\begin{eqnarray}
\dt\WWW_{ijkl}&=&  \Delta\WWW_{ijkl}+2(\BBB(\WWW)_{ijkl}-\BBB(\WWW)_{ijlk}-\BBB(\WWW)_{iljk}+\BBB(\WWW)_{ikjl})\\ \nonumber
&& +2\rho\RRR\WWW_{ijkl} -g^{pq}\big(\WWW_{pjkl}\RRR_{qi}+\WWW_{ipkl}\RRR_{qj}+\WWW_{ijpl}\RRR_{qk}+\WWW_{ijkp}\RRR_{ql}\big)\\ \nonumber
&& +\frac{2}{(n-2)^2}(\Ric^2\varowedge g)_{ijkl}+\frac{1}{(n-2)}(\Ric\varowedge\Ric)_{ijkl}\\ \nonumber
&& -\frac{2\RRR}{(n-2)^2}(\Ric\varowedge
g)_{ijkl}+\frac{\RRR^2-|\Ric|^2}{(n-1)(n-2)^2}(g\varowedge
g)_{ijkl}\nonumber \,,
\end{eqnarray}
where $\BBB(W)_{ijkl}=g^{pq}g^{rs}\WWW_{ipjr}\WWW_{kqls}$.
\end{prop}
\begin{proof}
By recalling the decomposition formula for the Weyl tensor~\eqref{weyl} we have
\begin{align*}
\dt\WWW=&\,\dt\Rm+\frac{1}{2(n-1)(n-2)}\Big(\dt\RRR g\varowedge g+2\dt
g\varowedge g\Big)-\frac{1}{n-2}\Big(\dt\Ric\varowedge
g+\Ric\varowedge \dt g\Big)\\
=&\,\M{L}_{II}+\M{L}_{0}\,,
\end{align*}  
where $\M{L}_{II}$ is the second order term in the curvatures and $\M{L}$ the $0$--th one.
We deal first with the higher order term; plugging in the evolution
equations of $\Rm, \Ric$ and $\RRR$ (Proposition~\ref{eqrmprop}) we get
\begin{eqnarray*}
\M{L}_{II}&=&\Delta\Rm-\rho(\nabla^2\RRR\varowedge g)+\frac{1-2(n-1)\rho}{2(n-1)(n-2)}\Delta\RRR g\varowedge g\\
&&-\frac{1}{n-2}(\Delta\Ric\varowedge g-(n-2)\rho\nabla^2\RRR\varowedge g-\rho\Delta\RRR g\varowedge g)\\
&=& \Delta\Rm+\frac{1-2(n-1)\rho+2(n-1)\rho}{2(n-1)(n-2)}\Delta\RRR g\varowedge g-\frac{1}{n-2}\Delta\Ric\varowedge g\\
&=&\Delta\WWW\,.
\end{eqnarray*}
Then we consider the lower order terms
\begin{eqnarray*}
(\M{L}_{0})_{ijkl}&=& 2(B(\Rm)_{ijkl}-B(\Rm)_{ijlk}-B(\Rm)_{iljk}+B(\Rm)_{ikjl})\\
&& - g^{pq}\big(\RRR_{pjkl}\RRR_{qi}+\RRR_{ipkl}\RRR_{qj}+\RRR_{ijpl}\RRR_{qk}+\RRR_{ijkp}\RRR_{ql}\big)\\
&& +2 \rho\RRR\Big(\WWW-\frac{1}{2(n-1)(n-2)}\RRR g\varowedge g+\frac{1}{n-2}\Ric\varowedge g\Big)_{ijkl}\\
&& +\frac{1}{2(n-1)(n-2)}(2|\Ric|^2g\varowedge g-2\rho\RRR^2 g\varowedge g-4\RRR\Ric\varowedge g+4\rho\RRR^2g\varowedge g)_{ijkl}\\
&& -\frac{1}{n-2}[2(\Rm*\Ric)\varowedge g-2\Ric^2\varowedge g-2\Ric\varowedge\Ric+2\rho\RRR\Ric\varowedge g]_{ijkl}\\
&=& 2(B(\Rm)_{ijkl}-B(\Rm)_{ijlk}-B(\Rm)_{iljk}+B(\Rm)_{ikjl})\\
&& - g^{pq}\big(\RRR_{pjkl}\RRR_{qi}+\RRR_{ipkl}\RRR_{qj}+\RRR_{ijpl}\RRR_{qk}+\RRR_{ijkp}\RRR_{ql}\big)+2\rho\RRR\WWW_{ijkl}\\
&& -\frac{2}{n-2}[(\Rm*\Ric)\varowedge g-\Ric^2\varowedge g-\Ric\varowedge\Ric]_{ijkl}\\
&& -\frac{2\RRR}{(n-1)(n-2)}(\Ric\varowedge g)_{ijkl}+\frac{|\Ric|^2}{(n-1)(n-2)}(g\varowedge g)_{ijkl}\,,
\end{eqnarray*}
where $(\Rm*\Ric)_{ab}=\RRR_{apbq}\RRR_{st}g^{ps}g^{qt}$ and $(\Ric^2)_{ab}=\RRR_{ap}\RRR_{bq}g^{pq}$.\\
Now we deal separately with every term containing the full curvature
$\Rm$, using its decomposition formula, expanding the Kulkarni--Nomizu
products and then contracting again. We have that
\begin{equation*}
[(g\varowedge g)*\Ric]_{ab}=2[\RRR g-\Ric]_{ab}\,, [(\Ric\varowedge g)*\Ric]_{ab}=[-2\Ric^2+\RRR\Ric+|\Ric|^2g]_{ab}\,.
\end{equation*}
Hence
\begin{eqnarray}\label{pezzo1}
(\Rm*\Ric)\varowedge g&=&(\WWW*\Ric)\varowedge g-\frac{2}{n-2}\Ric^2\varowedge g\\ \nonumber
&& +\frac{n\RRR}{(n-1)(n-2)}\Ric\varowedge g+\frac{(n-1)|\Ric|^2-\RRR^2}{(n-1)(n-2)}g\varowedge g\,.
\end{eqnarray}
Then
\begin{eqnarray*}
\RRR_{qi}\RRR_{pjkl}g^{pq}&=&\RRR_{qi}\Big(W_{pjkl}-\frac{\RRR}{(n-1)(n-2)}(g_{pk}g_{jl}-g_{pl}g_{jk})\Big)g^{pq}\\
&& +\frac{1}{n-2}\RRR_{qi}(\RRR_{pk}g_{jl}+\RRR_{jl}g_{pk}-\RRR_{pl}g_{jk}-\RRR_{jk}g_{pl})g^{pq}\\
&=&\RRR_{qi}\WWW_{pjkl}g^{pq}-\frac{\RRR}{(n-1)(n-2)}(\RRR_{ik}g_{jl}-\RRR_{il}g_{jk})\\
&& +\frac{1}{n-2}(\RRR^2_{ik}g_{jl}-\RRR^2_{il}g_{jk}+\RRR_{ik}\RRR_{jl}-\RRR_{il}\RRR_{jk})\,.
\end{eqnarray*}
Interchanging the index and using the symmetry properties we get
\begin{align}\label{pezzo2}
 g^{pq}\big(\RRR_{pjkl}\RRR_{qi}&\, +\RRR_{ipkl}\RRR_{qj}+\RRR_{ijpl}\RRR_{qk}+\RRR_{ijkp}\RRR_{ql}\big)\\ \nonumber
 =&\, g^{pq}\big(\WWW_{pjkl}\RRR_{qi}+\WWW_{ipkl}\RRR_{qj}+\WWW_{ijpl}\RRR_{qk}+\WWW_{ijkp}\RRR_{ql}\big)\\ \nonumber
+ &\,\frac{2}{n-2}(\Ric^2\varowedge g)_{ijkl}+\frac{2}{n-2}(\Ric\varowedge\Ric)_{ijkl}-\frac{2\RRR}{(n-1)(n-2)}(\Ric\varowedge g)_{ijkl}\,.
\end{align}
Finally the "$B$"--terms:
\begin{eqnarray*}
B(\Rm)_{abcd}&=&\Big(\WWW-\frac{\RRR}{2(n-1)(n-2)}g\varowedge g+\frac{1}{n-2}\Ric\varowedge g\Big)_{apbq}\\
&& \Big(\WWW-\frac{\RRR}{2(n-1)(n-2)}g\varowedge g+\frac{1}{n-2}\Ric\varowedge g\Big)_{csdt}g^{ps}g^{qt}
\end{eqnarray*}
\begin{equation*}
\big(\WWW_{apbq}(g\varowedge g)_{csdt}+(g\varowedge g)_{apbq}\WWW_{csdt}\big)g^{ps}g^{qt}=-2\WWW_{adbc}-2\WWW_{cbda}
\end{equation*}
\begin{align*}
\big(\WWW_{apbq}&\,(\Ric\varowedge g)_{csdt}+(\Ric\varowedge g)_{apbq}\WWW_{csdt}\big)g^{ps}g^{qt}=(\WWW*\Ric)_{ab}g_{cd}+(\WWW*\Ric)_{cd}g_{ab}\\
&\,- (\WWW_{cbdp}\RRR_{aq}+\WWW_{cpda}\RRR_{bq}+\WWW_{adbp}\RRR_{cq}+\WWW_{apbd}\RRR_{dq})g^{pq}
\end{align*}
\begin{eqnarray*}
(g\varowedge g)_{apbd}(g\varowedge g)_{csdt}g^{ps}g^{qt}=4\big((n-2)g_{ab}g_{cd}+g_{ac}g_{bd}\big)
\end{eqnarray*}
\begin{align*}
\big((\Ric\varowedge g)_{apbq}(g\varowedge g)_{csdt}&\,+(\Ric\varowedge g)_{csdt}(g\varowedge g)_{apbq}\big)g^{ps}g^{qt}\\
&\,=2\big( (n-4)\RRR_{ab}g_{cd}+(n-4)\RRR_{cd}g_{ab}+2\RRR_{ac}g_{bd}+2\RRR_{bd}g_{ac}\big)
\end{align*}
\begin{align*}
(\Ric\varowedge g)_{abpq}&\,(\Ric\varowedge g)_{csdt}g^{ps}g^{qt}= -2\RRR^2_{ab}g_{cd}-2\RRR^2_{cd}g_{ab}+\RRR^2_{ac}g_{bd}+\RRR^2_{bd}g_{ac}\\
&\,+ (n-4)\RRR_{ab}\RRR_{cd}+2\RRR_{ac}\RRR_{bd}+\RRR(\RRR_{ab}g_{cd}+\RRR_{cd}g_{ab})+|\Ric|^2g_{ab}g_{cd}
\end{align*}
Now, adding the same type quantities for the different index permutations and using the symmetry properties of $\WWW$ we obtain
\begin{align}\label{pezzo3}
B(\Rm)_{ijkl}&\,-B(\Rm)_{ijlk}-B(\Rm)_{iljk}+B(\Rm)_{ikjl}\\
=&\, B(\WWW)_{ijkl}-B(\WWW)_{ijlk}-B(\WWW)_{iljk}+B(\WWW)_{ikjl}\nonumber\\
&\,+\frac{1}{n-2}\big((\WWW*\Ric)\varowedge g\big)_{ijkl}-\frac{1}{(n-2)^2}(\Ric^2\varowedge g)_{ijkl}+\frac{1}{2(n-2)}(\Ric\varowedge\Ric)_{ijkl}\nonumber\\
&\, +\frac{\RRR}{(n-1)(n-2)^2}(\Ric\varowedge g)_{ijkl}+\Big(\frac{|\Ric|^2}{2(n-2)^2}-\frac{\RRR^2}{2(n-1)(n-2)^2}\Big)(g\varowedge g)_{ijkl}\,.\nonumber
\end{align}
We are ready to complete the computation of the $0$--th order term in the evolution equation, using the previous formulas~\eqref{pezzo1},~\eqref{pezzo2},~\eqref{pezzo3}
\begin{eqnarray*}
(\M{L}_{0})_{ijkl}&=& 2(B(\WWW)_{ijkl}-B(\WWW)_{ijlk}-B(\WWW)_{iljk}+B(\WWW)_{ikjl})+2\rho\RRR\WWW_{ijkl}\\ \nonumber
&& - g^{pq}\big(\WWW_{pjkl}\RRR_{qi}+\WWW_{ipkl}\RRR_{qj}+\WWW_{ijpl}\RRR_{qk}+\WWW_{ijkp}\RRR_{ql}\big)\\ \nonumber
&& +\frac{2}{(n-2)^2}(\Ric^2\varowedge g)_{ijkl}+\frac{1}{(n-2)}(\Ric\varowedge\Ric)_{ijkl}\\ \nonumber
&& -\frac{2\RRR}{(n-2)^2}(\Ric\varowedge g)_{ijkl}+\frac{\RRR^2-|\Ric|^2}{(n-1)(n-2)^2}(g\varowedge g)_{ijkl}
\end{eqnarray*}
\end{proof}

\subsection{Conditions preserved in dimension three}\ \\
In general dimension, it is very hard to find other curvature conditions preserved by the flow, and this is due principally to the complex structure of the reaction terms; for example in the evolution equation satisfied by the Ricci tensor~\eqref{eqrc}, the reaction terms involve the full curvature tensor. Therefore it is easier to restrict our attention to the three--dimensional case, in which the Weyl part of the Riemann tensor vanishes and all the geometric informations are encoded in the Ricci tensor.

In the special case of dimension three, we can use also the evolution equation~\eqref{evol_P} of the pull--back of the curvature operator to obtain more refined conditions preserved, because we can rewrite the ODE associated to the evolution of $\M{P}$ as a system of ODEs in the eigenvalues of $\M{P}$ that, by Proposition~\ref{proppm}, are nothing but the sectional curvatures of $\M{R}$. This point of view has been introduced for the Ricci flow by Hamilton in~\cite{hamilton10} and can be easily generalized to the RB flow as follows.
\begin{lemma}
If $n=3$, then $\M{P}_p$ has $3$ eigenvalues $\lambda,\mu,\nu$ and the ODE fiberwise associated to equation~\eqref{evol_P} can be written as the following system
\begin{equation}\label{eq_eigen}
\left\{
\begin{array}{lll}
\frac{d\lambda}{dt}&=&2\lambda^2+2\mu\nu-4\rho\lambda(\lambda+\mu+\nu)\,,\\
\frac{d\mu}{dt}&=&2\mu^2+2\lambda\nu-4\rho\mu(\lambda+\mu+\nu)\,,\\
\frac{d\nu}{dt}&=&2\nu^2+2\lambda\mu-4\rho\nu(\lambda+\mu+\nu)\,.
\end{array}
\right.
\end{equation}
In particular, if we assume $\lambda(0)\geq\mu(0)\geq\nu(0)$, then $\lambda(t)\geq\mu(t)\geq\nu(t)$ as long as the solution of the system exists.
\end{lemma} 
\begin{proof}
We can pointwise identify $V_p $ with an orthonormal frame of $\R^3$ with the standard basis. Then $\Lambda^2V_p\simeq\mathfrak{so}(3)$ with the standard structure constants and if an algebraic operator $Q_p$ is diagonal, both $Q_p^2$ and $Q_p^{\#}$ are diagonal with respect to the same basis (for the detailed computation of this fact, see~\cite[Chapter~6.4]{chknopf}). Hence the ODE $\frac{d}{dt} Q_p=F_p(Q_p)$, associated fiberwise to equation~\eqref{evol_P}, preserves the eigenvalues of $Q_p$, that is, if $Q_p(0)$ is diagonal with respect to an orthonormal basis, $Q_p(t)$ stays diagonal with respect to the same basis and the ODE can be rewritten as the system~\eqref{eq_eigen} in the eigenvalues.\\
To prove the last statement, we observe that
\begin{align*}
\frac{d}{dt}(\lambda-\mu)&=\,2(\lambda-\mu)\big((1-2\rho)(\lambda+\mu)-(1+2\rho)\nu\big)\\
\frac{d}{dt}(\mu-\nu)&=\,2(\mu-\nu)\big((1-2\rho)(\mu+\nu)-(1+2\rho)\lambda\big)\,.
\end{align*}
\end{proof}
\begin{rem}\label{eigen_proof}
We already proved that the differential operator in the evolution equation of $\M{P}$ is uniformly elliptic if $\rho<1/2(n-1)$, that is $\rho<1/4$ in dimension three. Therefore any geometric condition expressed in terms of the eigenvalues is preserved along the RB flow if the cone identified by the condition is closed, convex and preserved by the system~\eqref{eq_eigen}. 
\end{rem}

By using this method, we can prove 

 \begin{prop}\label{cond_3d}
 Let $(M,g(t))_{t\in[0,T)}$ be a compact, three--dimensional, solution of the RB flow~\eqref{eqflow}. If $\rho<1/4$, then
 \begin{itemize}
 \item[(i)] nonnegative Ricci curvature is preserved along the flow;
 \item[(ii)] nonnegative sectional curvature is preserved along the flow;
 \item[(iii)] the pinching inequality $\Ric\geq \varepsilon\RRR g$ is preserved along the flow for any $\varepsilon\leq 1/3$.
 \end{itemize}
 \end{prop}
\begin{proof}
(i) If $\Ric(g(0))\geq 0$, then $\Ric_{g(t)}\geq 0$.\\
The eigenvalues of $\Ric$ are the pairwise sums of the sectional curvatures, hence the condition is identified by the cone
\begin{equation*}
K_p=\{Q_p:(\mu+\nu)(Q_p)\geq 0 \}\,.
\end{equation*}
The closedness is obvious; in order to see that $K_p$ is convex, we observe that the greatest eigenvalue can be characterized by $\lambda(Q_p)=\max\{ Q_p(v,v):\,v\in V_p \,\,|v|_h=1\}$, hence it is convex. Then the function $Q_p\mapsto\mu(Q_p)+\nu(Q_p)=\tr(Q_p)-\lambda(Q_p)$ is concave and this implies that its superlevels are convex. By system~\eqref{eq_eigen} we obtain
\begin{equation*}
\frac{d}{dt}(\mu+\nu)=2\mu^2+2\nu^2+2\lambda(\mu+\nu)-4\rho(\mu+\nu)\tr(Q_p)\,.
\end{equation*}
There is the stationary solution corresponding to $\mu(0)=0=\nu(0)$. Otherwise, whenever $\mu(t_0)+\nu(t_0)=0$ with $\mu(t_0)\neq 0$ and $\nu(t_0)\neq 0$, $\frac{d}{dt}(\mu+\nu)(t_0)=2(\mu^2+\nu^2)(t_0)>0$, then $K$ is preserved.\\
(ii) If $\Sec(g(0))\geq 0$, then $\Sec_{g(t)}\geq 0$.\\
This condition is the non negativity of $\M{P}$, already proved in general dimension in Proposition~\ref{op_nonneg}, identified by the cone $K_p=\{Q_p:\, \nu(Q_p)\geq 0\}$, which is convex as superlevel of a concave function. We suppose that $\nu(t_0)=0$, then
\begin{equation*}
\frac{d}{dt}\nu(t_0)=2\lambda(t_0)\mu(t_0)\geq 0
\end{equation*} 
because the order between the eigenvalues is preserved and therefore $\lambda(t_0)\geq\mu(t_0)\geq 0$.\\
(iii) For every $\varepsilon\in (0,1/3]$, if $\Ric(g(0))-\varepsilon\RRR(g(0))g(0)\geq 0$, then $\Ric_{g(t)}-\varepsilon\RRR_{g(t)}g(t)\geq 0$.\\
Translating in terms of eigenvalues of $\M{P}$, the condition means $\mu(Q_p)+\nu(Q_p)-2\varepsilon\tr(Q_p)\geq 0$, that is $\lambda(Q_p)\leq\frac{1-2\varepsilon}{2\varepsilon}(\mu(Q_p)+\nu(Q_p))$, then the right cone is
\begin{equation*}
K_p=\{Q_p:\lambda(Q_p)-C(\varepsilon)(\mu(Q_p)+\nu(Q_p))\leq 0 \}\,,
\end{equation*}
where $C(\varepsilon)=\frac{1-2\varepsilon}{2\varepsilon}\in [1/2,+\infty)$. The defining function is the sum of two convex function, hence its sublevels are convex. Now, for $C=1/2$, that corresponds to $\varepsilon=1/3$, we have $\lambda(0)=\mu(0)=\nu(0)$ at each point of $M$,  that is the initial metric $g(0)$ has constant sectional curvature and this condition is preserved along the flow.\\
For $C>1/2$, we suppose $\lambda(t_0)=C(\mu(t_0)+\nu(t_0))$, then
\begin{align*}
\frac{d}{dt}(\lambda-C(&\mu+\nu))(t_0)=\,2\big[\lambda^2+\mu\nu-C(\mu^2+\nu^2+\lambda(\mu+\nu))-2\rho\tr(Q_p)(\lambda-C(\mu+\nu))\big](t_0)\\
&=\, 2\big[C^2(\mu(t_0)+\nu(t_0))^2+\mu(t_0)\nu(t_0)-C(\mu(t_0)^2+\nu(t_0)^2)-C^2(\mu(t_0)+\nu(t_0))^2\big]\\
&\leq\, (1-2C)(\mu(t_0)^2+\nu(t_0)^2)\leq 0\,.
\end{align*}
\end{proof}

\subsection{Hamilton--Ivey estimate}\ \\
A remarkable property of the three--dimensional Ricci flow is the pinching estimate, independently proved by Hamilton in~\cite{hamilton9} and Ivey in~\cite{ivey1}, which says that positive sectional curvature dominates negative sectional curvature during the Ricci flow, that is, if the initial metric $g_0$ has a negative sectional curvature somewhere, the Ricci flow starting at $g_0$ evolves the scalar curvature towards the positive semiaxis in future times, that means that there will be a greater (in absolute value) positive sectional curvature. \\
We have generalized the pinching estimate and some consequences for positive values of the parameter $\rho$. In the same notation used before, let $\lambda\geq\mu\geq\nu$ be the ordered eigenvalues of the curvature operator.
\begin{teo}[Hamilton-Ivey Estimate]\label{HIprop}
Let $(M,g(t))$ be a solution of the RB on a compact three--manifold such that the initial metric satisfies the normalizing assumption $\min_{p\in M}\nu_p(0)\geq -1$. If $\rho\in[0,1/6)$, then at any point $(p,t)$ where $\nu_p(t)<0$ the scalar curvature satisfies
\begin{equation}\label{HIeq}
\RRR\geq|\nu|\big(\log(|\nu|)+\log(1+2(1-6\rho)t)-3\big)
\end{equation} 
\end{teo}
\begin{proof}
We want to apply the Maximum Principle for time--dependent sets theorem~\ref{time_MP}, hence we need to express condition~\eqref{HIeq} in terms of a family of closed, convex, invariant subsets of $S^2(\Lambda^2V^*)$, where $(V,h(t),D(t))$ is the usual bundle isomorphism of the tangent bundle defined via Uhlenbeck's trick (Section~\ref{uhlenbeck}). Moreover, by~\cite[Lemma~10.11]{chowbookII}, we already know that, for any $t\in[0,T)$, the set
\begin{equation*}
K_p(t)=\left\{
\begin{array}{l}
Q_p : \tr(Q_p)\geq-\frac{3}{1+2(1-6\rho)t}\text{ and if }\nu(Q_p)\leq-\frac{1}{1+2(1-6\rho)t}\\
\text{then } \tr(Q_p)\geq|\nu(Q_p)|\big(\log(|\nu(Q_p)|)+\log(1+2(1-6\rho)t)-3\big)
\end{array}
\right\}
\end{equation*} 
defines a closed invariant subset of $S^2(\Lambda^2V^*)$. Since, for $\rho\in[0,1/6)$, $K(t)$ depends continuously on time, the space--time track of $K(t)$ is closed in $S^2(\Lambda^2V^*)$.\\
Now we show that $K_p(t)$ is convex for every $p\in M$ and $t\in[0,T)$. Following~\cite[Lemma~9.5]{chknopf}, we consider the map
\begin{equation*}
\Phi: S^2(\Lambda^2V^*_p)\to\R^2\,,\qquad \Phi(Q_p)=(|\nu(Q_p)|,\tr(Q_p))
\end{equation*}
Clearly, we have that $Q_p\in K_p(t)$ if and only if $\Phi(Q_p)\in A(t)$, where
\begin{equation*}
A(t)=\left\{
\begin{array}{l}
(x,y)\in\R^2 : y\geq-\frac{3}{1+2(1-6\rho)t}\,;\quad y\geq -3x\,;\\
\text{if }x\geq\frac{1}{1+2(1-6\rho)t}\text{ then } y\geq x\big(\log x+\log(1+2(1-6\rho)t)-3\big)
\end{array}
\right\}
\end{equation*}
is a convex subset of $\R^2$. Then in order to show that $K_p(t)$ is convex is sufficient to show that the segment between any two algebraic operators in $K_p(t)$ is sent by the map $\Phi$ into $A(t)$.\\ Therefore let $Q_p,Q'_p\in K_p(t)$, $s\in[0,1]$ and $Q_p(s)=sQ_p+(1-s)Q'_p$. About the first defining condition for $A(t)$, the trace is a linear functional, hence it is obviously fulfilled by $Q_p(s)$, while the second condition is satisfied by any algebraic operator.\\
The third condition is a bit tricky. If $\nu(Q_p)$, $\nu(Q'_p)>-\frac{1}{1+(1-6\rho)t}$, then the condition is empty for every point of the segment because $\nu$ is a concave function. By continuity we can assume w.l.o.g. that $\nu(Q_p(s))\leq-\frac{1}{1+(1-6\rho)t}$, for every $s\in[0,1]$, hence $x(Q_p(s))=-\nu(Q_p(s))$ is a convex function and $x(Q_p(s))\leq s x(Q_p)+(1-s)x(Q'_p)$. On the other hand the second condition implies that $x(Q_p(s))\geq -y(Q_p(s))/3=-\frac{1}{3}(s y(Q_p)+(1-s)y(Q'_p))$. Then $\Phi(Q_p(s))$ belongs to the trapezium of vertices 
$$
\Phi(Q_p),\,\, \Big(-\frac{1}{3} y(Q_p), y(Q_p)\Big),\,\, \Phi(Q'_p),\,\, \Big(-\frac{1}{3} y(Q'_p),y(Q'_p)\Big)\,,
$$
which is contained in $A(t)$, as its vertices are and $A(t)$ is convex.\\
Now we prove that $K(t)$ is preserved by the system~\eqref{eq_eigen}. By taking the sum of the three equations in the system (see also Remark~\ref{eigen_proof}) we get
\begin{equation*}
\frac{d}{dt}\tr(Q_p)\geq \frac{4}{3}(1-3\rho)\tr(Q_p)^2\,.
\end{equation*}
By hypothesis, $\nu(Q_p)(0)\geq -1$, hence $\tr(Q_p)(0)\geq -3$ for every $p\in M$ and by integrating the previous inequality,
\begin{equation*}
\tr(Q_p)(t)\geq-\frac{3}{1+4(1-3\rho)t}\geq -\frac{3}{1+2(1-6\rho)t}\,,
\end{equation*}
which holds for any $\rho\in[0,1/6)$.\\
In order to prove that the second inequality is preserved too, we consider, for every $p\in M$ such that $\nu(Q_p)(0)<0$, the function
\begin{equation}
f(t)=\frac{\tr(Q_p)}{-\nu(Q_p)}-\log(-\nu(Q_p))-\log(1+2(1-6\rho)t)
\end{equation}
and we compute its derivative along the flow.
\begin{align*}
\frac{d}{dt}f=&\,\frac{1}{\nu^2}\big[(-2\nu)\big(\lambda^2+\mu^2+\nu^2+\lambda\mu+\lambda\nu+\mu\nu-2\rho(\lambda+\mu+\nu)^2\big)\\
&\,+2(\lambda+\mu+\nu)\big(\nu^2+\lambda\mu-2\rho\nu(\lambda+\mu+\nu)\big)\big]\\
&\,-\frac{2}{\nu}\left(\nu^2+\lambda\mu-2\rho\nu(\lambda+\mu+\nu)\right)-\frac{2(1-6\rho)}{1+2(1-6\rho)t}\\
=&\,\frac{2}{\nu^2}\big[-\nu(\lambda^2+\mu^2+\lambda\mu)+\lambda\mu(\lambda+\mu)-\nu^3+2\rho\nu^2(\lambda+\mu+\nu)\big]-\frac{2(1-6\rho)}{1+2(1-6\rho)t}
\end{align*}
As in the case of the Ricci flow, it is easy to see that the quantity $-\nu(\lambda^2+\mu^2+\lambda\mu)+\lambda\mu(\lambda+\mu)$ is always nonnegative if $\nu<0$. In fact, if $\mu>0$ it is obvious, whereas if $\mu\leq0$ one has
$$
-\nu(\lambda^2+\mu^2+\lambda\mu)+\lambda\mu(\lambda+\mu) = (\mu-\nu)(\lambda^2+\mu^2+\lambda\mu)-\mu^{3} \geq 0 \,.
$$
Hence we get
\begin{equation}\label{HIcon}
\frac{d}{dt}f(t)\geq-2\nu+4\rho(\lambda+\mu+\nu)-\frac{2(1-6\rho)}{1+2(1-6\rho)t}
\end{equation}
If $\rho\geq 0$, since $\lambda+\mu+\nu\geq 3\nu$, we obtain
\begin{equation*}
\frac{d}{dt}f\geq -2(1-6\rho)\big(\nu+\frac{1}{1+2(1-6\rho)t}\big)\geq 0
\end{equation*}
whenever $\nu\leq-\frac{1}{1+2(1-6\rho)t}$ and $\rho\leq 1/6$.\\
Hence, if $(\lambda,\mu,\nu)$ is a solution of system~\eqref{eq_eigen} in $[0,T)$ with $(\lambda(0),\mu(0),\nu(0))\in K_p(0)$, we suppose that there is $t_1>0$ such that $\nu(t_1)<-\frac{1}{1+2(1-6\rho)t_1}$. Then, either $\nu(t)<-\frac{1}{1+2(1-6\rho)t}$ for any $t\in[0,t_1]$, either there exists $t_0<t_1$ such that $\nu(t_0)=-\frac{1}{1+2(1-6\rho)t_0}$ and $\nu(t)<-\frac{1}{1+2(1-6\rho)t}$ for any $t\in(t_0,t_1]$. In the first case, by hypothesis we obtain $f(0)\geq -3$ and $\frac{d}{dt}f(t)\geq 0$ for any $t\in[0,t_1]$, therefore $f(t_1)\geq-3$; in the second case $f(t_0)=\frac{(\lambda+\mu+\nu)(t_0)}{-\nu(t_0)}\geq -3$ and $\frac{d}{dt}f(t)\geq 0$ for any $t\in[t_0,t_1]$, therefore again $f(t_1)\geq-3$, which is equivalent to the second inequality.
\end{proof}
\begin{rem}
The extra term $4\rho(\lambda+\mu+\nu)$ on the key-equation~\eqref{HIcon} requires strong assumptions on the parameter $\rho$ since we have no information on the sign of the trace. However, combining equation~\eqref{HIcon} with Proposition~\ref{ancient_posR}, we can enlarge the range of $\rho$ to $[0,1/4)$, simply by dropping the extra term, nonnegative for ancient solutions and therefore conclude that an ancient solution to the RB flow on a compact three--manifold with bounded scalar curvature has nonnegative sectional curvature for any value of $\rho\in[0,1/4)$ (see~\cite[Corollary~9.8]{chknopf}).
\end{rem}

\begin{prop}\label{ancient_possec}
Let $(M, g(t))_{\,t\in(-\infty,0]}$ be a compact, three--dimensional, ancient solution of the RB flow~\eqref{eqflow} with uniformly bounded scalar curvature. If $\rho\in[0,1/4)$ then the sectional curvature is nonnegative.
\end{prop}

\section{Curvature estimates}
\subsection{Technical lemmas}
Before proving the curvature estimates for the RB flow, we need some
technical results.

First of all, we prove the proposition:
\begin{prop}\label{prop1}
Let $k \in \NN$, $p\in [1,+\infty]$ and $q\in [1,+\infty)$. There exists a constant
$C(n,k,p,q)$ such that for all $0\le j \le k$ and all tensor $T$
$$
\Vert \nabla^j T \Vert_{r_j} \le C \Vert T \Vert_p^{1-\frac{j}{k}} \Vert \nabla^k T
\Vert_{q}^{\frac{j}{k}}\,,
$$
where $\frac{1}{r_j}=\frac{1-\frac{j}{k}}{p} + \frac{\frac{j}{k}}{q}$.
\end{prop}

To prove this proposition, we need several lemmas.

\begin{lemma}\label{lem1}
Let $p\in [1,+\infty]$, $q\in [1,+\infty)$ and $r\in [2,+\infty)$ such that
$\frac{2}{r}=\frac{1}{p}+\frac{1}{q}$. There exists a constant $C(n,r)$ such that for all
tensor $T$
$$
\Vert \nabla T \Vert_r^2 \le C \Vert T \Vert_p \Vert \nabla^2 T
\Vert_q\,.
$$
\end{lemma}

\begin{proof}
\begin{eqnarray*}
\Vert \nabla T \Vert_r^r &=&
\int_M \left<\nabla T, \vert\nabla T\vert^{r-2} \nabla T
\right>\,d\mu_g\\
&=& -\int_M \left< T,\nabla\left( \vert\nabla T\vert^{r-2} \nabla T
   \right) \right> \,d\mu_g\\
&=& -\int_M \left< T,(r-2)\nabla^2 T \vert\nabla T\vert^{r-3} \nabla T
\right>\,d\mu_g 
- \int_M \left< T,\vert\nabla T\vert^{r-2} \nabla^2 T \right>\,d\mu_g \\
&\le& C \int_M \vert T\vert \vert \nabla^2 T \vert \vert \nabla
T\vert^{r-2}\,d\mu_g\\
&\le& C \Vert T\Vert_p \Vert \nabla^2 T \Vert_q \Vert \nabla
T\Vert_r^{r-2}\,,
\end{eqnarray*}
using H\"older's inequality with $\frac{r-2}{r}+\frac{1}{p}+\frac{1}{q}=1$. This ends the proof of this lemma.
\end{proof}

\begin{lemma}[Hamilton~\cite{hamilton1}, Corollary~12.5]\label{lem3}
Let $k \in \NN$. If $f \, : \, \{ 0,\dots,k \} \rightarrow \RR$ satisfies for all $0<j<k$
$$
f(j) \le C f(j-1)^{\frac{1}{2}}f(j+1)^{\frac{1}{2}}\,,
$$
where $C$ is a positive constant, then for all $0\le j \le k$
$$
f(j) \le C^{j(k-j)} f(0)^{1-\frac{j}{k}}f(k)^{\frac{j}{k}}\,.
$$
\end{lemma}

\begin{proof}[Proof of Proposition~\ref{prop1}]
We apply Lemma~\ref{lem3} with $f(j)= \Vert \nabla^j T \Vert_{r_j}$. Since $\frac{2}{r_j} = \frac{1}{r_{j-1}} + \frac{1}{r_{j+1}}$,
Lemma~\ref{lem1} shows that there exists $C(n,k,p,q)$ such that
$$
f(j) \le C f(j-1)^{\frac{1}{2}}f(j+1)^{\frac{1}{2}}\,,
$$
and then Lemma~\ref{lem3} gives Proposition~\ref{prop1}, since $r_0=p$ a $r_k=q$.
\end{proof}

\begin{lemma}\label{lem4}
For all tensors of the form $S \ast T$, there exists $C$ depending on the dimension and the coefficients
in the expression such that
$$
\vert S \ast T \vert \le C \vert S\vert \vert T \vert\,.
$$
\end{lemma}

\begin{proof}
By Cauchy--Schwarz inequality, $(\tr_gT)^2=\left(g^{\alpha\beta}T_{\alpha\beta}\right)^2 \le n T_{\alpha\beta}T^{\alpha\beta}=n|T|^2$. Then
$$
\vert S \ast T \vert \le C(n) \vert S \otimes T \otimes g^{\otimes j}
\otimes (g^{-1})^{\otimes k} \le C(n) n^{\frac{j+k}{2}}\vert S\vert
\vert T \vert\,.
$$
\end{proof}

Let $k \in \NN$, and set, for a tensor $T$,
$\displaystyle F_g(T) = \sum_{j+l=k, \, j, l \ge 0} \nabla^j T \ast
\nabla^l T \ast \nabla^k T$.

\begin{lemma}\label{lem5}
Let $k\in \NN$. Let $p\in [2,+\infty]$ and $q \in [2,+\infty)$ such that $\frac{1}{p}+\frac{2}{q}=1$. There exists
$C(n,k,p,q,F)$ such that for all tensor $T$,
$$
\int_M \vert F_g(T)  \vert \,d\mu_g \le C \Vert T \Vert_p \Vert \nabla^kT
\Vert_q^2\,.
$$
\end{lemma}

\begin{proof}
Let us consider one term in $F_g(T)$ that can be written $\nabla^j T \ast \nabla^l T \ast \nabla^k T$, $j,l \ge 0$. We set 
$$
\frac{1}{r_j} = \frac{1-\frac{j}{k}}{p} + \frac{\frac{j}{k}}{q}\,,\qquad \frac{1}{r_l} = \frac{1-\frac{l}{k}}{p} + \frac{\frac{l}{k}}{q}\,.
$$
Since $\frac{1}{r_j} + \frac{1}{r_j} + \frac{1}{q}=1$, by Lemma~\ref{lem4} and H\"older's inequality we have
\begin{eqnarray*}
\int_M \vert \nabla^j T \ast \nabla^l T \ast \nabla^k T \vert\,d\mu_g
 &\le& C' \int_M \vert \nabla^j T \vert \vert\nabla^l T \vert \vert \nabla^k T \vert\,d\mu_g \\
&\le& C' \Vert \nabla^j T \Vert_{r_j} \Vert\nabla^l T \Vert_{r_l} \Vert\nabla^k T \Vert_{q} \,,
\end{eqnarray*}
Then, by applying Proposition~\ref{prop1} to the first two factors, we get
$$
\int_M \vert \nabla^j T \ast \nabla^l T \ast \nabla^k T \vert\,d\mu_g
\le C \Vert T \Vert_p \Vert \nabla^k T \Vert_q^2\,.
$$
The result follows since $F_g(T)$ is a linear combination of such terms.
\end{proof}
\subsection{Curvature estimates}
\begin{teo}\label{teocurvest}
Assume $\rho < \frac{1}{2(n-1)}$.
If $g(t)$ is a compact solution of the RB flow for $t\in [0,T)$ such that
$$
\sup_{(x,t) \in M\times[0,T)} \vert \Rm (x,t)\vert \le K\,,
$$
then for all $k\in \NN$ there exists a constant $C(n,\rho,k,K,T)$ such that for all  $t\in (0,T]$
$$
\Vert \nabla^k \Rm_{g(t)} \Vert_2^2 \le \frac{C}{t^k} \sup_{t
  \in [0,T)} \Vert \Rm_{g(t)} \Vert_2^2\,.
$$
\end{teo}

\begin{proof}
A direct computation gives
\begin{eqnarray*}
\dt\vert \Rm \vert^2 &=&
\Delta (\vert \Rm \vert^2) - 2 \vert\nabla \Rm\vert^2 -8\rho \RRR_{ij}\nabla^i\nabla^j \RRR + \Rm \ast \Rm \ast \Rm \\
\dt \RRR^2 &=&
(1-2(n-1)\rho)\Delta(\RRR^2) - 2(1-2(n-1)\rho) \vert\nabla \RRR\vert^2 +4\RRR \vert \Ric\vert^2 - 4\rho \RRR^3\,.
\end{eqnarray*}
It follows that
\begin{eqnarray*}
\frac{d}{dt}\int_M\vert \Rm \vert^2\,d\mu_g &=&
 - 2 \int_M\vert\nabla \Rm\vert^2\,d\mu_g
 -8\rho \int_M \RRR_{ij}\nabla^i\nabla^j \RRR\,d\mu_g\\
&& + \int_M \Rm \ast \Rm \ast \Rm \,d\mu_g\\
\frac{d}{dt} \int_M \RRR^2\,d\mu_g &=&
- 2(1-2(n-1)\rho) \int_M \vert\nabla \RRR\vert^2\,d\mu_g
 +\int_M \Rm \ast \Rm \ast \Rm\,d\mu_g\,.
\end{eqnarray*}
Now we want to compute $\int_M \RRR_{ij}\nabla^i\nabla^j \RRR\,d\mu_g$. Using Bianchi identity we have:
$$
\int_M \RRR_{ij}\nabla^i\nabla^j \RRR\,d\mu_g=-\frac{1}{2} \int_M \vert\nabla
\RRR\vert^2\,d\mu_g\,.
$$
We conclude that
\begin{eqnarray*}
\frac{d}{dt}\int_M\vert \Rm \vert^2\,d\mu_g &=&
 - 2 \int_M\vert\nabla \Rm\vert^2\,d\mu_g +4\rho \int_M \vert\nabla \RRR\vert^2\,d\mu_g\\
 && + \int_M \Rm \ast \Rm \ast \Rm \,d\mu_g\\
\frac{d}{dt} \int_M \RRR^2\,d\mu_g &=&
- 2(1-2(n-1)\rho) \int_M \vert\nabla \RRR\vert^2\,d\mu_g+\int_M \Rm \ast \Rm \ast \Rm\,d\mu_g\,.
\end{eqnarray*}
As we did before, a straightforward computation gives:
\begin{eqnarray*}
\frac{d}{dt}\int_M\vert \nabla^k \Rm \vert^2\,d\mu_g &=&
 - 2 \int_M\vert\nabla^{k+1} \Rm\vert^2\,d\mu_g +4\rho \int_M \vert\nabla^{k+1} \RRR\vert^2\,d\mu_g \\
 &&+ \sum_{j+l=k, \, j,l\ge 0}\int_M \nabla^{j}\Rm \ast \nabla^{l}\Rm \ast \nabla^{k}\Rm\,d\mu_g \\
\frac{d}{dt} \int_M \vert \nabla^k \RRR \vert^2\,d\mu_g &=&
- 2(1-2(n-1)\rho) \int_M \vert\nabla^{k+1} \RRR\vert^2\,d\mu_g \\
&&+ \sum_{j+l=k, \, j,l\ge 0}\int_M \nabla^{j}\Rm \ast \nabla^{l}\Rm \ast \nabla^{k}\Rm\,d\mu_g\,.
\end{eqnarray*}

Consider 
$${\mathcal A}_k := \int_M\vert\nabla^k \Rm\vert^2\,d\mu_g + \frac{4\vert \rho\vert}{(1-2(n-1)\rho)}\int_M \vert\nabla^k \RRR\vert^2\,d\mu_g$$
and set $f_k(t):= \displaystyle\sum_{j=0}^k \frac{\beta^j t^j}{j!}{\mathcal A}_j$, where $\beta :=\min (1, 1-2(n-1)\rho)$. 
We have
\begin{equation}\label{der_fk}
f'_k(t) = \sum_{j=0}^{k-1} \frac{\beta^j t^j}{j!}\left({\mathcal A}'_j + \beta{\mathcal A}_{j+1}\right)
+ \frac{\beta^k t^k}{k!} {\mathcal A}'_k 
\end{equation}
We have by a direct computation, for any $j$:
\begin{eqnarray*}
{\mathcal A}'_j + \beta {\mathcal A}_{j+1} &= & (-2+\beta) \Vert\nabla^{j+1}\Rm\Vert_2^2 + \Big(4\rho-8|\rho|+ \frac{4\beta|\rho|}{1-2(n-1)\rho}\Big)\Vert\nabla^{j+1}\RRR\Vert_2^2\\
&& +\sum_{i+l=j, \, i,l\ge 0}\int_M \nabla^{i}\Rm \ast \nabla^{l}\Rm \ast \nabla^{j}\Rm\,d\mu_g\,.
\end{eqnarray*}
We need to estimate $\sum_{i+l=j, \, i,l\ge 0}\int_M \nabla^{i}\Rm
\ast \nabla^{l}\Rm \ast \nabla^{j}\Rm\,d\mu_g$. For this we use
Lemma~\ref{lem5} with $p=+\infty$ and $q=2$:
$$
\sum_{i+l=j, \, i,l\ge 0}\int_M \nabla^{i}\Rm \ast \nabla^{l}\Rm \ast \nabla^{j}\Rm\,d\mu_g \le
C \Vert \Rm \Vert_\infty \Vert \nabla^{j}\Rm \Vert_2^2\,.
$$
Using Proposition~\ref{prop1}, with $k=j+1$ we get
\begin{equation*}
\sum_{i+l=k, \, i,l\ge 0}\int_M \nabla^{i}\Rm \ast \nabla^{l}\Rm \ast \nabla^{j}\Rm\,d\mu_g \leq \,
C \Vert \Rm \Vert_\infty (\Vert \Rm \Vert_2^2)^{\frac{1}{j+1}}(\Vert\nabla^{j+1} \Rm \Vert_{2}^2)^{\frac{j}{j+1}}\,.
\end{equation*}
Now we apply Young's inequality $ab\leq \frac{a^p}{p}+\frac{b^q}{q}$, where
$$
a=C\Vert \Rm \Vert_\infty (\Vert \Rm \Vert_2^2)^{\frac{1}{j+1}}\,,\quad b=(\Vert\nabla^{j+1} \Rm \Vert_{2}^2)^{\frac{j}{j+1}}
$$
and $p=j+1$, $q=\frac{j+1}{j}$, we use the hypothesis on the boundedness of $\Vert \Rm \Vert_\infty$ and we obtain
\begin{equation*}
\sum_{i+l=j, \, i,l\ge 0}\int_M \nabla^{i}\Rm \ast \nabla^{l}\Rm \ast \nabla^{j}\Rm\,d\mu_g \leq \,
C'(n,\rho,j,K)\Vert \Rm \Vert_2^2+ \Vert \nabla^{j+1}\Rm \Vert_2^2\,.
\end{equation*}
Putting this last inequality in the previous computation, we obtain
\begin{eqnarray*}
{\mathcal A}'_j + \beta {\mathcal A}_{j+1} & \leq & (-1+\beta) \Vert\nabla^{j+1}\Rm\Vert_2^2 + \Big(4\rho-8|\rho|+ \frac{4\beta|\rho|}{1-2(n-1)\rho}\Big)\Vert\nabla^{j+1}\RRR\Vert_2^2\\
&& + C'(n,\rho,j,K)\Vert \Rm \Vert_2^2\\
& \leq & C'(n,\rho,j,K)\Vert \Rm \Vert_2^2\,,
\end{eqnarray*}
where we use the facts that $-1+\beta \le 0$ and $4\rho -8\vert\rho\vert +  \frac{4\vert \rho\vert \beta}{(1-2(n-1)\rho)} \le 0$.
The same estimates holds for the last term in equation~\eqref{der_fk}, since
$$
\M{A}'_k\leq \M{A}'_k+\beta \M{A}_{k+1}\leq C'(n,\rho, k, K) \Vert \Rm \Vert_2^2
$$
Therefore
\begin{eqnarray*}
f'_k(t) &\leq & \sum_{j=0}^k \frac{\beta^j t^j}{j!} C'(n,\rho,j,K)\Vert \Rm \Vert_2^2\\
& \leq & \overline{C}(n,\rho,k, K) \Vert \Rm \Vert_2^2 (e^{\beta t}-1)\\
&\leq & \widetilde{C}(n,\rho,k,K,T)\Vert \Rm \Vert_2^2\,.
\end{eqnarray*}
Since $f_k(0)=\M{A}_0\leq C(\rho, n)\Vert \Rm \Vert_2^2 $, by integrating the previous inequality we finally get
\begin{align*}
\Vert \nabla^k\Rm\Vert_2^2 &\,\le {\mathcal A}_k \le \frac{k!}{\beta^k t^k}
f_k(t)  
\le \frac{\widehat{C}}{t^k}\Bigl[f_k(0)+\widetilde{C}t\Vert \Rm \Vert_2^2\Bigl]\\
&\,\le\frac{\widehat{C}\big[C(\rho,n)+\widetilde{C}t\big]}{t^k} \Vert \Rm \Vert_2^2
\le \frac{C}{t^k} \Vert \Rm \Vert_{2}^2\,,
\end{align*}
which concludes the proof of the theorem.
\end{proof}

\subsection{Long time existence}\ \\
In this section we will prove the following result.

\begin{teo}\label{longtime} Assume $\rho < \frac{1}{2(n-1)}$.
If $g(t)$ is a compact solution of the RB flow on a maximal time interval $[0,T)$, $T<+\infty$, then
$$
\limsup_{t\rightarrow T} \max_{M} |\mathrm{\Rm}(\,\cdot\,,t)| \,=\, +\infty\,.
$$
\end{teo}

\begin{proof} This proof follows exactly the one given by Hamilton for the Ricci flow (see~\cite[Section 14]{hamilton1}. First of all we observe that, if the Riemann tensor is uniformly bounded as $t\to T$ and $T<+\infty$, then also its $L^2$--norm is uniformly bounded, because from the previous computations, for $\M{A}_0=\Vert \Rm \Vert_2^2 + \frac{4\vert \rho \vert}{(1-2(n-1)\rho)}\Vert\RRR\Vert_2^2$ we have $\M{A}'_0\leq C \M{A}_0$.\\
Then, by Theorem~\ref{teocurvest}, we get, for any $j\in \NN$
$$
\Vert \nabla^j\Rm\Vert_2^2\leq C_j\,.
$$
Now, by using the interpolation inequalities in Proposition~\ref{prop1} with $p=\infty$, $q=2$, we immediately get the estimates
$$
\Vert \nabla^{j} \Rm\Vert_{\frac{2k}{j}} \,\leq\, C_{j,k}\,,
$$
for all $j\in\NN$ and $k\geq j$. Therefore, by interpolation the same result holds for a generic exponent $r$, with the constant that depends on $j$ and $r$.\\
Now, let $E_{j}:=|\nabla^{j}\Rm|^{2}$. Then, for all $r<+\infty$ we have
$$
\int_{M}\big(|E_{j}|^{r}+|\nabla E_{j}|^{r}\big)\,d\mu_g  \,\leq \, C_{j,r}'\,.
$$
Thus, by Sobolev inequality, if $r>j$, one has
$$
\max_{M}|E_{j}|^{r} \,\leq\, C_{t} \int_{M} \big(|E_{j}|^{r}+|\nabla E_{j}|^{r}\big)\,d\mu_g \,.
$$
Notice that the constant $C_{t}$ depends on the metric $g(t)$, but it
does not depend on the derivatives of $g(t)$. Moreover,
from~\cite[Lemma~14.2]{hamilton1}, it follows that the metrics are all
equivalent. Hence, the constant $C_{t}$ is uniformly bounded as
$t\rightarrow T$ and, from the previous estimates, it follows that, if
$|\Rm|\leq C$ on $M\times[0,T)]$, for every $j\in\NN$ one has
$$
\max_{M} |\nabla^{j} \Rm| \,\leq \, C_{j}\,,
$$
where the constant $C_{j}$ depends only on the initial value of the
metric and the constant $C$.

Arguing now as in~\cite[Section~14]{hamilton1}, it follows that the metrics $g(t)$ converge to some limit metric $g(T)$ in the $C^{\infty}$ topology (with all their time/space ordinary partial derivatives, once written in local coordinates), hence, we can restart the flow with this initial metric $g(T)$, obtaining a smooth flow in some larger time interval $[0,T+\delta)$, in contradiction with the fact that $T$ was the maximal time of smooth existence. This
completes the proof of Theorem~\ref{longtime}. 
\end{proof}

\bibliographystyle{amsplain}
\bibliography{biblio}

\end{document}